\documentclass[preprint,11pt]{elsarticle}
\usepackage[toc,page]{appendix}
\usepackage{graphicx} 
\usepackage{amsmath,amsfonts,amsthm,amssymb,amscd}
\usepackage{lipsum}
\usepackage{amsfonts}
\usepackage{graphicx}
\usepackage{epstopdf}
\usepackage{algorithmic}
\usepackage{appendix}
\usepackage{subfig}
\usepackage{graphicx}
\usepackage{color}
\usepackage{cancel}

\usepackage{hyperref}
\usepackage{xcolor}
\usepackage{cleveref}
\usepackage{bm}
\usepackage{soul}
\usepackage[normalem]{ulem}

\usepackage{enumitem}
\usepackage{fullpage}
\usepackage{listings}
\usepackage{dsfont}
\usepackage{booktabs} 
\usepackage{multirow} 
\usepackage{diagbox}
\usepackage{tablefootnote}

\newcommand{\cN}{\mathcal{N}}
\newcommand{\U}{\mathcal{U}}
\newcommand{\A}{\mathcal{A}}

\newcommand{\G}{\mathcal{G}}
\newcommand{\cF}{\mathcal{F}}
\newcommand{\cK}{\mathcal{K}}

\newcommand{\C}{\mathbb{C}}

\newcommand{\PP}{\mathcal{P}}

\newcommand{\I}{I}
\newcommand{\dd}{\mathrm{d} }
\newcommand{\R}{\mathbb{R}}

\newcommand{\Z}{\mathbb{Z}}

\newcommand{\cL}{\mathcal{L}}

\newcommand{\cR}{\mathcal{R}}

\newcommand{\bigO}{\mathcal{O}}

\DeclareMathOperator*{\argmin}{arg\,min}

\newtheorem{lemma}{Lemma}

\newtheorem{theorem}{Theorem}

\definecolor{darkred}{rgb}{.6,0,0}
\definecolor{darkblue}{rgb}{0,0,.7}
\definecolor{darkgreen}{rgb}{0,.7,0}
\definecolor{update}{rgb}{0.8,,0}

\makeatletter
\def\ps@pprintTitle{%
   \let\@oddhead\@empty
   \let\@evenhead\@empty
   \def\@oddfoot{\reset@font\hfil\thepage\hfil}
   \let\@evenfoot\@oddfoot
}
\makeatother

\begin{document}

\begin{abstract}
Surrogate models are critical for accelerating computationally expensive simulations in science and engineering, particularly for solving parametric partial differential equations (PDEs). Developing practical surrogate models poses significant challenges, particularly in handling geometrically complex and variable domains, which are often discretized as point clouds.
In this work, we systematically investigate the formulation of neural operators---
maps between infinite-dimensional function spaces---on point clouds to better handle complex and variable geometries while mitigating discretization effects. We introduce the Point Cloud Neural Operator (PCNO), designed to efficiently approximate solution maps of parametric PDEs on such domains.
We evaluate the performance of PCNO on a range of pedagogical PDE problems, focusing on aspects such as boundary layers, adaptively meshed point clouds, and variable domains with topological variations. Its practicality is further demonstrated through three-dimensional applications, such as predicting pressure loads on various vehicle types and
simulating the inflation process of intricate parachute structures.

\end{abstract}
\begin{keyword}
    Parametric Partial Differential Equations, Surrogate Modeling,  Neural Networks,  Computational Mechanics, Complex Geometries.
\end{keyword}

\begin{frontmatter}

  \title{Point Cloud Neural Operator for Parametric PDEs on Complex and Variable Geometries}

\author[rvt1]{Chenyu Zeng\corref{cor1}}
  \ead{zengyu@stu.pku.edu.cn}
  
  \author[rvt1]{Yanshu Zhang\corref{cor1}}
  \ead{yanshu@stu.pku.edu.cn}
  
  \author[rvt1]{Jiayi Zhou\corref{cor1}}
  \ead{jiayi22@stu.pku.edu.cn}
  
  \author[rvt1]{Yuhan Wang}
  \ead{wyhan@stu.pku.edu.cn}
  
  \author[rvt1]{Zilin Wang}
  \ead{wangzilin@stu.pku.edu.cn}
  
  \author[rvt2]{Yuhao Liu}
  \ead{yh-liu21@mails.tsinghua.edu.cn}
  
  \author[rvt1,rvt3]{Lei Wu}
  \ead{leiwu@math.pku.edu.cn}

  \author[rvt4,rvt3]{Daniel~Zhengyu~Huang\corref{cor2}}
  \ead{huangdz@bicmr.pku.edu.cn}
  \cortext[cor1]{Equal contribution, listed in alphabetical order.}
  \cortext[cor2]{Corresponding author.}

  \address[rvt1]{School of Mathematical Sciences, Peking University, Beijing, China}
  \address[rvt2]{Department of Mathematical Sciences, Tsinghua University, Beijing, China}
  \address[rvt3]{Center for Machine Learning Research, Peking University, Beijing, China}
  \address[rvt4]{Beijing International Center for Mathematical Research, Peking University, Beijing, China}

\end{frontmatter}

\section{Introduction}
The development of efficient and accurate approximations for  computationally expensive simulations, also known as surrogate models, is crucial for addressing many query problems in computational science and engineering. 
These simulations often arise from the numerical solution of partial differential equations (PDEs).
Applications include engineering design \cite{jameson1988aerodynamic,mack2007surrogate,martins2013multidisciplinary,bendsoe2013topology,amsallem2015design,economon2016su2,du2021rapid},  where timely evaluations of candidate designs are critical; 
model predictive control \cite{garcia1989model,peitz2018survey,korda2018linear,kapteyn2021probabilistic,mcclellan2022physics}, which demands computational efficiency to enable real-time decision-making; and uncertainty quantification \cite{gilks1995markov,kennedy2001bayesian,xiu2002wiener,stuart2010inverse,schobi2015polynomial,ghanem2017handbook,zhu2019physics,chen2024efficient}, where thousands of forward simulations are required to reliably assess uncertainty. In this work, we specifically focus on the development of surrogate models for parametric PDEs.

Consider an abstract parametric PDE defined on $\Omega \subset \R^{d}$
\begin{equation}
\label{eq:abstract-pde}
    \cR(u, a) = 0, 
\end{equation}
where $\cR$ denotes a generalized differential operator, $a:\R^d \rightarrow \R^{d_a}$ represents the parameter function, and 
 $u(x): \R^d \rightarrow \R^{d_u}$ represents the solution function. Solving \eqref{eq:abstract-pde} gives the solution map 
\begin{align}
\label{eq:G-continuous-intro}
    \mathcal{G^{\dagger}} : (a, \Omega) \mapsto u.
\end{align}
The goal of surrogate modeling is to accelerate the inference of the solution map in \eqref{eq:G-continuous-intro}. 
However, developing efficient and accurate PDE-based surrogate models presents several \textbf{key challenges}. 
First, the inputs and outputs of the model, $a$ and $u$, are often high-dimensional or, in some cases, infinite-dimensional, particularly when full-field PDE solutions are required.
Second, the computational domain $\Omega$ is often complex, and in applications such as engineering design or scenarios involving variable geometries, both the computational domain and its discretization can vary significantly.  
The limitations and progress of various surrogate modeling approaches in addressing these challenges are reviewed in \cref{ssec:literature}.
To overcome these challenges, this study focuses on improving neural network-based surrogate models, given their proven effectiveness in approximating high-dimensional functions and functionals~\cite{barron1993universal, lecun2015deep, vaswani2017attention, kingma2013auto, song2020score, achiam2023gpt}. Additionally, modern software packages such as PyTorch~\cite{paszke2019pytorch} and JAX~\cite{bradbury2018jax} offer automatic differentiation, which facilitates further applications, including optimization-based engineering design and gradient-based sampling algorithms for uncertainty quantification.

The objective of this work is to develop a neural operator~\cite{zhu2018bayesian,khoo2019switchnet,li2020fourier,lu2021learning} based surrogate model that employs a neural network to approximate $\G^{\dagger}$ described in \eqref{eq:G-continuous-intro} in a manner that is based on the following distinct \textbf{contributions}:
\begin{enumerate}
\item We systematically investigate the formulation of neural operators, which combines integral and differential operators~\cite{li2020fourier,liu2024neural}, on point clouds. By designing at the continuous level and subsequently leveraging traditional numerical methods for discretization, this approach ensures robust handling of complex and variable geometries with arbitrary discretizations, achieving small mesh-size-dependent errors and significantly reducing the need for hyperparameter tuning.

\item We introduce the Point Cloud Neural Operator (PCNO), a surrogate modeling framework designed for parametric PDEs on complex and variable geometries. It features point permutation invariance and universal approximation capabilities, while maintaining linear complexity with respect to the number of points.

\item We demonstrate that PCNO effectively learns operators between function spaces on complex and variable geometries for a variety of PDE problems, including those involving boundary layers, adaptive mesh data, and topological variations. Furthermore, it achieves outstanding performance in large-scale three-dimensional applications, such as predicting pressure loads on different vehicles and simulating the inflation process of intricate parachute structures.
\end{enumerate}

\subsection{Literature Review}
\label{ssec:literature}

Surrogate models aim to approximate the solution map described in \eqref{eq:G-continuous-intro}. A prominent class of such models includes projection-based reduced-order models~\cite{willcox2002balanced,antoulas2005approximation,quarteroni2014reduced,benner2015survey,benner2017model}, where orthogonal bases, such as those derived from Proper Orthogonal Decomposition (POD)~\cite{berkooz1993proper,golub1965calculating} or Koopman modes \cite{koopman1931hamiltonian,mezic2004comparison}, are constructed. The original PDE is projected onto the low-dimensional subspace spanned by these bases, enabling efficient computation of the solution. This projection process can be implemented intrusively by incorporating the governing PDE information from \eqref{eq:abstract-pde}, leading to a reliable and accurate reduced-order model~\cite{chaturantabut2010nonlinear,carlberg2013gnat,carlberg2011efficient}. Alternatively, reduced models can be learned directly from data, resulting in data-driven approaches. These methods encompass a wide range of techniques, including linear~\cite{schmid2010dynamic,rowley2009spectral,kutz2016dynamic} and quadrature~\cite{peherstorfer2016data,qian2020lift} approximations, Gaussian process models~\cite{williams2015data,guo2018reduced}, and fully connected neural networks~\cite{hesthaven2018non,lee2020model,bhattacharya2021model,barnett2022neural}, all of which aim to approximate the operator map in \eqref{eq:G-continuous-intro}. 
Despite their effectiveness, these approaches often struggle to adapt to changes in the computational domain, $\Omega$. The reduced basis typically cannot accommodate variations in discretization or domain geometry, limiting their applicability to the variable computational domain considered in this work.

Neural networks hold great promise for developing efficient surrogate models for PDE problems, due to their ability to approximate high-dimensional functions effectively. Two primary strategies are commonly employed to create neural network-based surrogate models capable of handling complex and variable computational domains: the \textbf{domain deformation strategy} and the \textbf{point cloud strategy}.

The first approach involves constructing a reference domain $\Omega_{\rm ref}$ and a parametric deformation map $\psi$ that transforms the reference domain into parameterized computational domains $\Omega$. This concept, inspired by the Lagrangian description of problems, is widely used in adaptive mesh methods \cite{tang2003adaptive,huang2010adaptive}, where the reference mesh dynamically evolves in accordance with the underlying physics.
By reformulating the PDE on the fixed reference domain, the geometry information is encoded in the deformation map, and the mapping \eqref{eq:G-continuous-intro} becomes
\begin{align*}
    \mathcal{G^{\dagger}} : (a \circ \psi, \psi) \mapsto u \circ \psi \quad \textrm{ on } \quad \Omega_{\rm ref},
\end{align*}
where the functions are all defined on $\Omega_{\rm ref}$.
Combining this domain deformation strategy with the Fourier Neural Operator (FNO) results in GeoFNO \cite{li2023fourier}. This approach has also been integrated with the physics-informed neural network (PINN) framework \cite{gao2021phygeonet, cao2024solving} and the DeepONet framework \cite{yin2024dimon, xiao2024learning} to improve accuracy and training efficiency.
However, constructing the parametric deformation map $\psi$ is challenging, particularly when the parameterized computational domains $\Omega$ exhibit varying topologies (See numerical examples in \cref{ssec:airfoil_flap}).

An alternative approach avoids the need for constructing parametric deformation maps by representing complex and variable computational domains $\Omega$ as point clouds, leveraging geometric learning techniques to directly process these domains. 
Point clouds inherently capture the geometry of 
$\Omega$, enabling methods such as PointNet \cite{qi2017pointnet} and its extensions \cite{qi2017pointnet++, li2018pointcnn} to extract features for classification and segmentation. 
Frameworks like DeepONet \cite{lu2021learning, goswami2022deep} and transformers \cite{cao2021choose,junfengpositional} have also been extended to use point cloud representations for developing surrogate models with complex but fixed computational domains (See \cref{ssec:transformer}).
Furthermore, graph neural networks (GNNs) \cite{scarselli2008graph, kipf2016semi} provide an alternative for encoding point cloud data, integrating geometric features and connectivity information. GNNs offer a powerful framework for constructing PDE-based surrogate models \cite{pfaff2020learning, liu2024laflownet,gao2025generative,shen2025vortexnet}, but their performance, is highly dependent on mesh resolution and necessitates careful adjustment of the message-passing depth whenever the resolution changes. This can be mitigated through mesh-independent averaging aggregation, which leads to graph neural operator \cite{li2020neural, li2020multipole} (See \cref{ssec:GNO}). Another strategy involves embedding point clouds within a bounding box and using zero padding to reformulate the problem on a regular grid, allowing for the direct application of Fourier Neural Operator~\cite{liu2024domain, li2024geometry} or convolution-based neural networks. However, this embedding generally requires careful interpolation or extrapolation based on the mesh resolution.
Beyond the direct use of point clouds, geometric features can be further enriched using techniques like signed distance functions~\cite{he2024geom, ye2024pdeformer,duvall2025discretization}, which encode the geometry as a scalar field, or advanced geometric encoders such as Sinusoidal Representation Networks~\cite{sitzmann2020implicit}. These approaches extract detailed geometric features and have been applied in PDE-based surrogate modeling~\cite{serrano2023operator, wu2024transolver}.
However, their generalization performance across varying discretizations or resolutions remains unclear.
The present work aims to systematically explore the formulation of neural operators for point clouds, with a focus on achieving invariance to or mitigating the impact of discretization and resolution. These properties enable robust handling of complex and variable computational domains with arbitrary discretizations, thereby reducing the need for hyperparameter tuning. Notably, all tests presented in this work utilize the same architecture and identical hyperparameters, such as the number of channels and network depth.

\subsection{Organization}
\label{ssec:over}
The remainder of the paper is organized as follows.  In \cref{sec:neural-operator}, we provide a comprehensive overview of neural operators. \Cref{sec:pcno} introduces the proposed Point Cloud Neural Operator (PCNO), followed by the related theoretical analysis in  \cref{sec:theory}. Numerical experiments are presented in \cref{sec:numerics} and concluding remarks are provided in \cref{sec:conclusion}.

\section{Neural Operator}
\label{sec:neural-operator}

In this section, we provide a concise overview of neural operators, which are neural network-based surrogate models designed to approximate mappings where both inputs and outputs are functions.
The primary focus of this work is the mapping
\begin{equation}
\label{eq:G-continuous}
    \mathcal{G^{\dagger}} : (a, \Omega) \mapsto u.
\end{equation} 
Here  $a$ represents the parameter function (which includes source terms, boundary conditions, and other factors), $\Omega$ denotes the computational domain in $\R^d$, and $u$ represents the solution function. Generally, the parameter function space is denoted as $\A =\{a: \Omega \rightarrow \R^{d_a}\}$, and the solution function space as $\U =\{u:\Omega\rightarrow \R^{d_u}\}$; both are separable Banach spaces.
Neural operator approach aims to approximate $\mathcal{G^{\dagger}}$ using a parametric neural network
$G_{\theta}$, where  $\theta \in \Theta$ is assumed to be a finite dimensional parameter. To train neural operators, we consider the following standard data scenario. The input data $\{a_i,\Omega_i\}_{i=1}^{n}$ is sampled from certain distribution $\mu$, and the corresponding outputs are given by $\{u_i = \G^{\dagger}(a_i,\Omega_i)\}_{i=1}^n$. 
Given a cost function $c: \U \times \U \rightarrow \R$, and the dataset $\{a_i, \Omega_i,  u_i\}_{i=1}^{n}$, the optimal $\theta^{*}$ is obtained by solving the optimization problem:
\begin{equation}
\label{eq:training}
    \theta^* = \argmin \frac{1}{n}\sum_{i=1}^n c(\G_{\theta}(a_i, \Omega_i), u_i).
\end{equation}
The goal is for 
$G_{\theta^{*}}$ to approximate 
$\mathcal{G^{\dagger}}$ effectively within the distribution $\mu$.

In practical applications, the computational domain $\Omega$ is often represented by a point cloud $X$, with the connectivity between points  defining the mesh structure.
The functions $a$ and $u$ are typically available only at the discrete points of $X$. Neural operators, however, are  designed to approximate the operator \eqref{eq:G-continuous} at the continuous level, as a mapping between function spaces, rather than at the discrete points. This framework grants neural operators with discretization invariant properties, allowing the learned operator to generalize robustly across different discretizations.

In general, to approximate the operator in \eqref{eq:G-continuous}, each layer of neural operator $\cL$  combines pointwise linear functions, integral operators and a nonlinear pointwise activation function, which maps the input function $f_{\rm in}: \Omega \rightarrow \R^{d_{\rm in}}$ to the output function $f_{\rm out}: \Omega \rightarrow \R^{d_{\rm out}}$, such as 
\begin{equation}
\begin{split}
    \label{eq:integral-operator}
    &\cL : f_{\rm in} \mapsto f_{\rm out}, \\
    &f_{\rm out}(x) =  \sigma \Bigl(W^{l} f_{\rm in}(x) + b + \int_{\Omega} \kappa\bigl(x, y, f_{\rm in}\bigr)  W^{v} f_{\rm in}(y) dy\Bigr).
\end{split}
\end{equation}
Here, $W^{l}\in\R^{d_{\rm out}\times d_{\rm in}}$ and $b\in\R^{d_{\rm out}}$ define a pointwise local linear function.\footnote{ In general, the integral operator with continuous kernel between two infinite dimensional spaces (e.g., continuous function space $C(\Omega)$) is a compact operator, however identity mapping between two infinite dimensional spaces is not \cite[Section 2]{kress1989linear}. The pointwise local linear function is able to represent such identity mappings and hence is necessary to improve the performance.}
Note that $d_{\rm in}$ and $d_{\rm out}$ refer to the number of channels of the input and output functions, respectively. Such a linear function is also used to increase the number of channels $d_a$ of the input function $a$ in the initial lifting layer. This is crucial because a higher-dimensional channel space can capture more complex features and make the relationships more tractable~\cite{koopman1931hamiltonian,mezic2005spectral,qian2020lift,huang2024operator}. In the final projection layer, a linear function projects the output back to the desired number of channels $d_u$.
$W^{v} \in \R^{d_{\rm out}\times d_{\rm in}}$ parameterizes the integral operator, which is defined by the kernel $\kappa(x, y, f_{\rm in}):\Omega \times \Omega \times (\Omega \rightarrow \R^{d_{\rm in}}) \rightarrow \R$. 
The integral operator in \eqref{eq:integral-operator} can be seen as an analogy to matrix multiplication in the context of neural networks operating on finite-dimensional vector spaces. 
The pointwise nonlinear activation function $\sigma$, such as the Gaussian Error Linear Unit (\texttt{GeLU})\cite{hendrycks2016gaussian}, acts on each point and each channel. It is worth noting that the operator $\cL$ described in \eqref{eq:integral-operator} provides a nonlinear parametric mapping between function spaces. 
The design of the neural operator is carried out at the continuous level and then discretized. 

Next, we review several specific integral operators utilized in the Fourier Neural Operator~(FNO), Graph Neural Operator (GNO), and Transformer. 

\subsection{Fourier  Neural  Operator}
\label{ssec:FNO}
The  Fourier Neural Operator~\cite{li2020fourier,kovachki2023neural,nelsen2021random,kossaifi2024library} considers a series of Fourier kernel functions defined as 
$\kappa_k(x, y) = e^{2\pi i k \cdot (x - y)} \quad k\in \Z^d$,
which can approximate positive semidefinite translationally invariant kernels \cite{bochner1933monotone,rahimi2007random}. This leads to the following Fourier integral operator:
\footnote{
The formula for the Fourier integral operator in the literature~\cite{li2020fourier,de2022cost} is typically written as $f_{\rm out} = \cF^{-1} \Bigl(R \cdot \bigl(\cF f_{\rm in}\bigr)\Bigr) $, where $\cF$ denotes the Fourier transform for each channel. This maps the periodic function $f_{\rm in}: \Omega \rightarrow \R^{d_{\rm in}}$ to Fourier coefficients, $\bigl(\cF f_{\rm in} \bigr) (k) \in \C^{d_{\rm in}}$, generally truncated at $k_{\rm max}$ modes. $R$ define $k_{\rm max}$ complex-valued matrices $R(k) \in \R^{d_{\rm out} \times d_{\rm in}}$, which is applied on each of the Fourier modes as $R(k) \bigl(\cF f_{\rm in}\bigr) (k) \in \R^{d_{\rm out}}$.
Finally, $\cF^{-1}$ denotes the inverse Fourier transform for each channel, resulting in $f_{\rm out}: \Omega \rightarrow \R^{d_{\rm out}}$. In our equivalent formulation in \eqref{eq:FNO}, we separate each Fourier mode, denote $W_{k}^v = R(k)$, and use the fact that $W_{k}^v \bigl(\cF f_{\rm in}\bigr)(k) = \bigr(\cF (W_{k}^v  f_{\rm in})\bigr)(k)$.
}
\begin{align}
\label{eq:FNO}
    f_{\rm out}(x) = \sum_{k}\int_{\Omega} \kappa_k(x, y)  W_k^{v} f_{\rm in}(y) dy,
\end{align}
where each frequency $k$ corresponds to a complex parameter matrix $W_k^{v}\in \C^{d_{\rm out}\times d_{\rm in}}$. The integral \eqref{eq:FNO} can be efficiently evaluated using the Fast Fourier Transform, achieving quasi-linear complexity with respect to the number of mesh points when $\Omega$ is a hypercube discretized on a structured grid.  Using the fact that the Fourier kernel has a low-rank representation $\kappa_k(x, y) = \phi_k(x) \overline{\phi_k(y)}$, where $\phi_k(x) = e^{2\pi i k \cdot x}$, and focusing only on low-frequency modes $k$, the Fourier Neural Operator restricted to low frequencies can generalize beyond regular grids \cite{lingsch2023beyond} with linear complexity using the following formulation
\begin{align}
\label{eq:kernel-sep}
\int_{\Omega} \kappa_k(x, y)  W_k^{v} f_{\rm in}(y) dy = 
\phi(x) \int_{\Omega}  \overline{\phi(y)} W_k^{v} f_{\rm in}(y) dy = 
e^{2\pi i k \cdot x} \int_{\Omega}  e^{-2\pi i k \cdot y} W_k^{v} f_{\rm in}(y) dy.  
\end{align}
More generally, any serials of kernels $\kappa_k(x, y) = \phi_k(x) \phi_k(y)$ with a low-rank representation can achieve linear computational complexity. For example, in \cite{chen2023learning}, the authors study $\phi_k(x)$ as either the $k$-th eigenvector of the Laplacian operator on $\Omega$ or the POD basis learned from the training data.

\subsection{Graph Neural Operator}
\label{ssec:GNO}
In the graph neural network framework \cite{scarselli2008graph,kipf2016semi,pfaff2020learning,liu2024laflownet}, information is propagated locally through the connectivity of the point cloud to learn the mapping in \eqref{eq:G-continuous}. This results in a discretization-dependent approach that requires careful tuning of the message-passing depth.
This mechanism is generalized to the graph neural operator \cite{li2020neural,li2020multipole}, by replacing the averaging aggregation as the integral operator with a local kernel, where  $\kappa(x,y) = 0$ when $\lVert x - y \rVert > r$. The integral can be evaluated as
\begin{align}
\label{eq:kernel-local}
\int_{\Omega} \kappa(x, y)  W^{v} f_{\rm in}(y) dy  = 
\int_{ \Omega ~ \bigcap ~ \lVert x - y \rVert \leq r }  \kappa(x, y) W^{v} f_{\rm in}(y) dy,
\end{align}
with $\bigO(m)$ integration points sampled from the region $\Omega ~ \bigcap ~ \lVert x - y \rVert \leq r$. The complexity of evaluating the integral  \eqref{eq:kernel-local} becomes linear with respect to the number of mesh points.  However, its performance might be sensitive to the hyperparameters $r$ and $m$, which may vary depending on the mesh resolution.

\subsection{Transformer}
\label{ssec:transformer}
 In the Transformer framework’s attention mechanism~\cite{vaswani2017attention,calvello2024continuum}, the function values $f_{\rm in}(x) \in \R^{d_{\rm in}}$ are interpreted as word embeddings.
 Given a point cloud $\{x^{(i)}\}_{i=1}^{N}$ representing the domain $\Omega$, the attention mechanism with a softmax operator computes the output at each point $x^{(i)}$ as
\begin{align}
    f_{\rm out}(x^{(i)}) = \sum_{j=1}^{N} 
    \frac{e^{\langle W^q f_{\rm in}(x^{(i)}) , W^k f_{\rm in}(x^{(j)})\rangle / \sqrt{d_{\rm out}}}}{\sum_{j^{'}=1}^{N} e^{\langle W^q f_{\rm in}(x^{(i)}), W^k f_{\rm in}(x^{(j^{'})}) \rangle/\sqrt{d_{\rm out}}}}   W^v f_{\rm in}(x^{(j)}),
\end{align}
where $W^q, W^k, W^v \in \R^{d_{\rm out} \times d_{\rm in}}$ are the query, key, and value matrices, and $\langle\cdot, \cdot\rangle$ denotes the inner product in $\R^{d_{\rm out}}$.
 In the continuous limit, as the number of points approaches infinity and assuming uniform sampling from the domain $\Omega$, the attention mechanism leads to 
 \begin{align}
    f_{\rm out}(x) =\int_{\Omega} 
    \frac{e^{\langle W^q f_{\rm in}(x) , W^k f_{\rm in}(y)\rangle / \sqrt{d_{\rm out}}}}{\int_{\Omega} e^{\langle W^q f_{\rm in}(x), W^k f_{\rm in}(z) \rangle/\sqrt{d_{\rm out}}} dz}   W^v f_{\rm in}(y) dy,
\end{align}
 which gives rise to the following nonlinear parametric kernel:
 \begin{align}
 \label{eq:transformer}
    \kappa(x, y, f_{\rm in}) = \frac{ e^{\langle W^q f_{\rm in}(x) , W^k f_{\rm in}(y) \rangle/\sqrt{d_{\rm out}}} }{\int_{\Omega} e^{\langle W^q f_{\rm in}(x) ,  W^k f_{\rm in}(z) \rangle/\sqrt{d_{\rm out}}} dz}.
\end{align}
Positional information can also be used in the kernel to develop neural operators~\cite{junfengpositional}.
However, evaluating the integral operator  \eqref{eq:integral-operator} with this kernel \eqref{eq:transformer} typically involves quadratic complexity with respect to the number of mesh points.  Removing the softmax operation \cite{cao2021choose,guo2022transformer} simplifies the kernel to 
\begin{align}
    \kappa(x, y, f_{\rm in}) = \langle W^q f_{\rm in}(x) , W^k f_{\rm in}(y) \rangle,
\end{align}
which enables a low-rank representation and reduces the computational cost to linear complexity.

\section{Neural Operator on Point Cloud}
\label{sec:pcno}
In this section, we introduce a novel neural operator specifically designed to learn the mapping described in \eqref{eq:G-continuous}, with a particular focus on handling complex and variable geometries, $\Omega \subset \R^{d}$. Notably, $\Omega$ can represent a low-dimensional manifold within  $\R^d$, such as the surface of a vehicle in $\R^3$ (See \cref{ssec:vehicle}).
To account for both global and local effects, and inspired by the work of \cite{liu2024neural}, our approach separates integral and differential operators, and leverages traditional numerical methods for their discretizations.
This design ensures that the neural operator, initially defined in function space, can robustly handle complex and variable computational domains with arbitrary discretizations, and that the subsequent discretization results in small mesh-size-dependent errors.
Since the geometry $\Omega$ is typically represented as a point cloud with connectivity information, $X = \{x^{(i)}\}_{i=1}^{N}$,  where the number of points $N$ may vary across datasets,  we refer to our approach as the Point Cloud Neural Operator (PCNO), with further details provided in the following sections.

\subsection{Integral Operator}
\label{ssec:integral-op}
 We associate the geometry $\Omega$ with a density function $\rho(x; \Omega)$ defined on $\Omega$.  The integral operator is defined as:
\begin{equation}
\begin{split}
    \label{eq:integral-operator-rho}
    &f_{\rm out}(x) =  \int_{\Omega} \kappa\bigl(x, y, f_{\rm in}\bigr)  W^{v} f_{\rm in}(y) \rho(y; \Omega) dy,
\end{split}
\end{equation}
where $\kappa$ is the kernel, and $W^{v}\in\R^{d_{\rm out}\times d_{\rm in}}$ is a parameter matrix. 
The density function $\rho(x; \Omega)$ ensures that the integral does not inappropriately scale the function by the domain volume $|\Omega|$, which enhances stability during training and facilitates numerical integration.

The density function $\rho(x; \Omega)$ can be set as a uniform density
$\rho(x; \Omega) = \frac{1}{|\Omega|}$ defined on $\Omega$. In this case, $\rho(x; \Omega)$ acts as the indicator function of $\Omega$. Given a point cloud $\{x^{(i)}\}_{i=1}^N$ representing the computational domain $\Omega$, the integral operator  \eqref{eq:integral-operator-rho} can be  numerically approximated as:
\begin{equation}
\begin{split}
    \label{eq:integral-operator-numerics-uniform}
    &f_{\rm out}(x) \approx \frac{1}{|\Omega|} \sum_{i=1}^{N} \kappa\bigl(x, x^{(i)}, f_{\rm in}\bigr)  W^{v} f_{\rm in}(x^{(i)})  \dd\Omega_i,
\end{split}
\end{equation}
where $\dd\Omega_i$ represents the estimated mesh size associated with point $x^{(i)}$, derived from the point cloud. Details on the computation of $\dd\Omega_i$ are provided in \ref{sec:mesh}.

In practice, the mesh is often adaptively generated, with mesh points more densely distributed in regions of higher importance (See \cref{ssec:airfoil_flap}). 
Accordingly, $\rho(y; \Omega)$ can be set to represent the density from which the point cloud is sampled.
This allows the integral operator in \eqref{eq:integral-operator-rho} to assign greater weight to these critical regions through the term $\rho(y; \Omega)$. 
Given the point cloud, the density can be estimated as $\rho(x^{(i)}; \Omega) = \frac{1}{N\dd\Omega_i}$, which satisfies the normalization condition  
$\int_{\Omega} \rho(x,\Omega)\dd x := \sum_{i=1}^N \rho(x^{(i)}; \Omega) \dd\Omega_i = 1$. 
Then the integral operator in \eqref{eq:integral-operator-rho} can be numerically approximated as:
\begin{equation}
\begin{split}
    \label{eq:integral-operator-numerics}
    &f_{\rm out}(x) \approx \sum_{i=1}^{N}\kappa\bigl(x, x^{(i)}, f_{\rm in}\bigr)  W^{v} f_{\rm in}(x^{(i)}) \rho(x^{(i)}; \Omega) \dd\Omega_i = \frac{1}{N} \sum_{i=1}^{N}\kappa\bigl(x, x^{(i)}, f_{\rm in}\bigr)  W^{v} f_{\rm in}(x^{(i)}). 
\end{split}
\end{equation}
This expression is equivalent to the Monte Carlo method for numerical integration.

When $\rho(x; \Omega)$ is defined as the point cloud density, it automatically incorporates mesh adaptivity information (if adaptivity is present), thereby improving the accuracy of the results (See \cref{ssec:airfoil_flap}). 
However, this approach requires the adaptive strategy to be consistent across all cases, as each $\Omega$ is associated with a unique $\rho(x;\Omega)$. If the adaptive strategy is not universal, the uniform density should be used instead, as this approach is free from such limitations (See \cref{ssec:adv}).

\subsection{Differential Operator}
\label{ssec:differential-op}
We further define a differential operator to capture local effects. 
The gradient $\nabla f_{\rm in} : \R^d \rightarrow \R^{d \times d_{\rm in}}$ can be computed independently for each of the $d_{\rm in}$ channels. Without loss of generality, we assume $d_{\rm in} = 1$ for the following discussion. The gradient is estimated via a least-squares approach implemented as a single message-passing step.
Given the connectivity of the point cloud (See \ref{sec:mesh} for details on computing connectivity), let the neighbors of $x$ be denoted as $x^{(1)}, x^{(2)},\cdots x^{(m)}$. The gradient $\nabla f_{\rm in}(x)$ is estimated by solving the least-squares problem:
\begin{equation}
\begin{split}
&A(x)  \nabla f_{\rm in}(x) = b_f(x) \\
& \textrm{ where } A(x)  = \begin{pmatrix}
x^{(1)} - x\\
x^{(2)} - x\\
\vdots\\
x^{(m)} - x\\
\end{pmatrix} \in \R^{m \times d} \quad \textrm{and}\quad 
b_f(x) =  
 \begin{pmatrix}
f_{\rm in}(x^{(1)}) - f_{\rm in}(x)\\
f_{\rm in}(x^{(2)}) - f_{\rm in}(x)\\
\vdots\\
f_{\rm in}(x^{(m)}) - f_{\rm in}(x)\\
\end{pmatrix} \in \R^{m}.
\end{split}
\end{equation}
The least-squares problem can be solved by computing the pseudoinverse of the matrix $A(x)^{\dagger} \in \R^{d \times m} $. Special consideration is required when $\Omega$ is a submanifold of $\R^d$. Let $d^{'} \leq d$ denote the intrinsic dimension of $\Omega$. For example,  lines have dimension $d^{'} = 1$, and surfaces have dimension $d^{'} = 2$. The truncated singular value decomposition of $A(x)$ with rank $d^{'}$ is given by
\begin{equation}
    A(x) = U\Sigma V^T, \textrm{ where } U\in \R^{m \times d^{'}},\quad  \Sigma \in \R^{d^{'} \times d^{'}}, \quad V^T\in \R^{d^{'} \times d}.
\end{equation}
The pseudoinverse is then computed as $A(x)^{\dagger} = V\Sigma^{-1} U^T$. 
When $\Omega$ is a submanifold of $\R^d$, the estimated gradient $\nabla f_{\rm in}(x) = A(x)^{\dagger}b_f(x)$ satisfies 
\begin{equation}
    n(x)^T \cdot \nabla f_{\rm in}(x)  = 0
\end{equation} 
for any normal vector $n(x)$ of $\Omega$ at $x$, since $A(x) \cdot n(x) = 0$. Consequently, the gradient $\nabla f_{\rm in}(x)$ lies within the tangent space of $\Omega$ at $x$.
For efficient implementation, we use the fact that  $A(x)^{\dagger}$ depends solely on the point cloud and its connectivity, allowing it to be preprocessed before neural operator training.
Additionally, the gradient $\nabla f_{\rm in}(x) = A(x)^{\dagger}b_f(x)$ can be efficiently assembled using a message-passing procedure. 
Specifically, the $i$-th column of $A(x)^{\dagger}$, denoted as $A_i(x)^{\dagger} \in \R^{d}$, can be stored at the edge connecting $x$ and $x^{(i)}$. For any function $f_{\rm in }$, its gradient at $x$ is computed by aggregating information from its neighbors, as follows
\begin{equation}
\label{eq:gradient_operator}
    \nabla f_{\rm in}(x) = \sum_{i: x^{(i)} \textrm{ is a neighbor of } x} A_i(x)^{\dagger} \bigl(f_{\rm in}(x^{(i)}) - f_{\rm in}(x)\bigr).
\end{equation}

While the gradient is an intrinsic property of $f_{\rm in}$, and is independent of connectivity or mesh, discontinuities in  $f_{\rm in}$ result in infinite gradients. In such cases,  the gradient estimation scales as $\mathcal{O}(1/\Delta x)$ where $\Delta x$ represents the local mesh size around $x$. To address this issue, we apply the \texttt{SoftSign} activation function to each gradient component: 
\begin{equation}
\label{eq:smoothed-gradient}
    \widetilde{\nabla} f_{\rm in}(x) = \texttt{SoftSign}(\nabla f_{\rm in}(x) ) \qquad \textrm{ where }\qquad \texttt{SoftSign}(x) = \frac{x}{1 + |x|}.
\end{equation}
The smoothed gradient $\widetilde{\nabla}$ preserves gradient information while serving as an indicator for sharp or discontinuous local features in $f_{\rm in}$.

\subsection{Point Cloud Neural Operator}
Finally, by combining the aforementioned integral operator and differential operator, we introduce the following point cloud neural layer, which maps the input function $f_{\rm in}: \Omega \rightarrow \R^{d_{\rm in}}$ to the output function $f_{\rm out}: \Omega \rightarrow \R^{d_{\rm out}}$:
\begin{equation}
\begin{split}
\label{eq:pcno-layer}
    &\cL : f_{\rm in} \mapsto f_{\rm out}, \\
    &f_{\rm out}(x) = 
    \sigma \Bigl( W^{l} f_{\rm in}(x) + b + \sum_{k}\int_{\Omega} e^{2\pi i \frac{k}{L} \cdot(x - y) }  W_{k}^{v} f_{\rm in}(y) \rho(y; \Omega)dy  +  W^{g} \widetilde{\nabla} f_{\rm in}(x) \Bigr),
\end{split}
\end{equation}
where the components of the layer are defined as follows:
\begin{itemize}
\item The first term represents a pointwise local linear function parameterized by $W^{l}\in\R^{d_{\rm out}\times d_{\rm in}}$ and $b\in\R^{d_{\rm out}}$. 
\item The second term represents the integral operator. 
In this work, we focus on the Fourier kernel discussed in \cref{ssec:FNO}: $\kappa\bigl(x, y, f_{\rm in}\bigr) = e^{2\pi i \frac{k}{L} \cdot x} e^{- 2\pi i \frac{k}{L} \cdot y}$. Here $k \in \Z^{d}$ denotes the frequency, $L \in \R_{+}^d$ denotes the learnable length scales, and each frequency $k$ is associated with a complex-valued parameter matrix $W_k^{v}\in \C^{d_{\rm out}\times d_{\rm in}}$.  
\item The third term represents the differential operator, parameterized by $W^{g}\in \R^{d_{\rm out}\times (d \cdot d_{\rm in})}$. Here, the smoothed gradient $\widetilde{\nabla} f_{\rm in}(x) \in \R^{d \times d_{\rm in}}$ in \eqref{eq:smoothed-gradient} is flattened into a vector in $\R^{d \cdot d_{\rm in}}$ for computation.
\end{itemize}
The nonlinear activation function $\sigma$ used in this work is \texttt{GeLU} function~\cite{hendrycks2016gaussian}.

As shown in \cref{fig:architecture}, the Point Cloud Neural Operator $\G_\theta$ considered in this work consists of several point cloud neural layers \eqref{eq:pcno-layer}. 
The input function is defined as 
\begin{equation}
\label{eq:tildea}
    \tilde{a} : \Omega \rightarrow \R^{d_a + d + 1}, ~~\textrm{ where } \tilde{a}(x) = \begin{bmatrix}
    a(x)\\
    x \\
    \rho(x; \Omega)
\end{bmatrix},
\end{equation} 
and consists of three components. The first component is the parameter function $a: \R^d \rightarrow \R^{d_a}$, which includes information such as source terms and boundary conditions. The latter two components are the identity function $\mathrm{id}: \R^d \rightarrow \R^d$ and the density function $\rho(x;\Omega)$, both of which encode geometric information about the computational domain $\Omega$ and the point cloud.
At the implementation level, the input  $\boldsymbol{\tilde{a}} \in \R^{N \times (d_a + d + 1)}$ consists of evaluations of the input functions $\tilde{a}$ at the point cloud  $X = \{x^{(i)}\}_{i=1}^{N}$.

\begin{figure}
    \centering
    \includegraphics[width=0.95\linewidth]{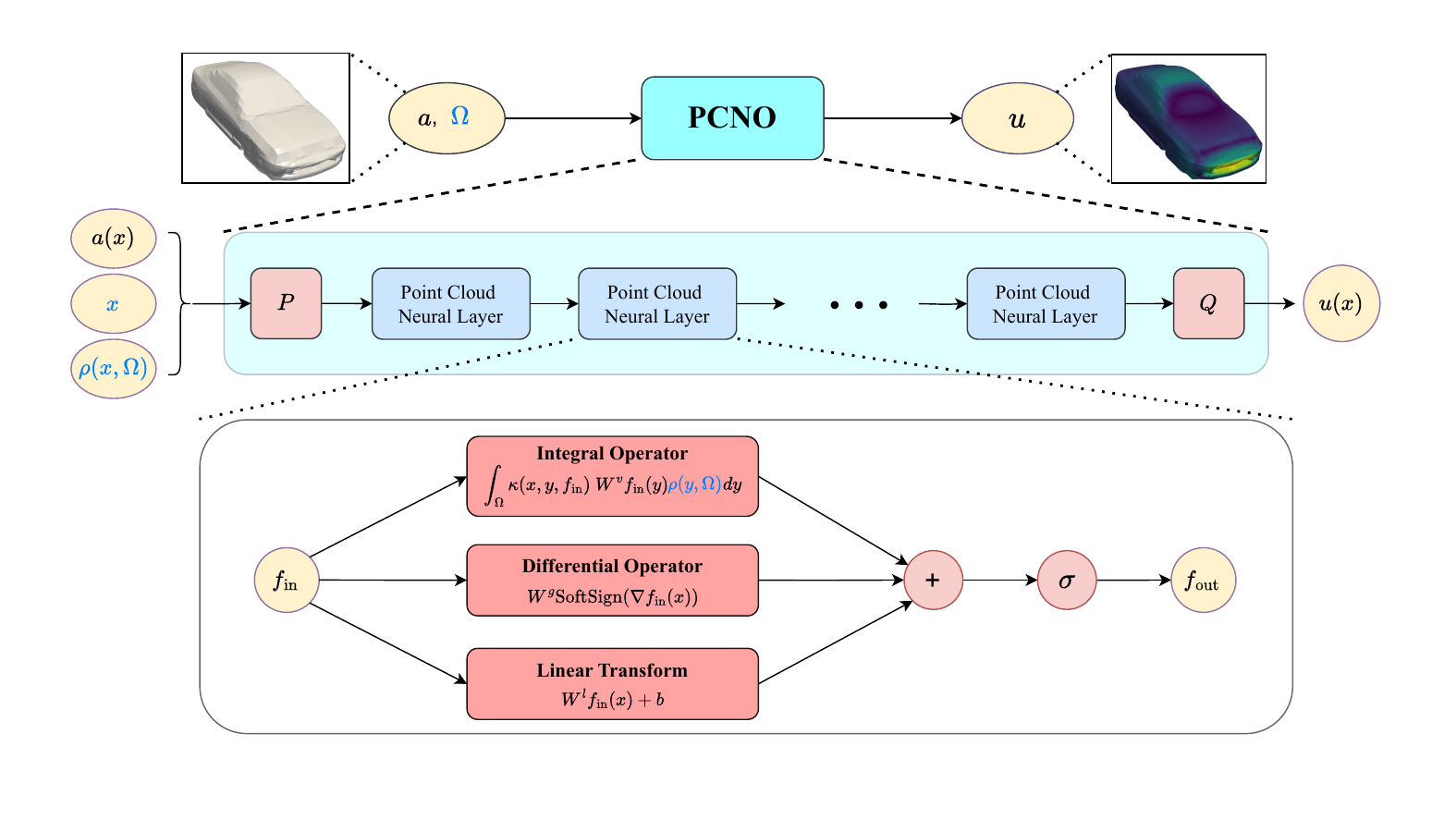}
    \caption{
    Architecture of the Point Cloud Neural Operator (PCNO): Starting with input function $[a(x),x,\rho(x,\Omega)]$, it is first lifted to a higher-dimensional channel space using a pointwise neural network  $P$. Several point cloud neural layers are then applied, followed by a projection back to the target dimensional space using a pointwise neural network  $Q$. The final output is $u$.}
    \label{fig:architecture}
\end{figure}

The first layer is a lifting layer, which only contains a pointwise neural network:
\begin{equation}
\begin{split}
\label{eq:pcno-layer-lifting}
f_{\rm out}(x) = 
    \sigma \Bigl( W^{l} \tilde{a}(x) + b\Bigr).
\end{split}
\end{equation}
This is crucial for lifting the channel space from  $d_{a} + d + 1$ to a higher-dimensional channel space $d_{\rm out}$, enabling the capture of more complex features and making the underlying relationships more tractable~\cite{koopman1931hamiltonian,mezic2005spectral,qian2020lift}.
In the final projection layer, the output function is projected back to the desired number of channels, 
$d_u$. In the present work, the projection layer consists of two consecutive pointwise neural network layers, given by:
\begin{equation}
\begin{split}
\label{eq:pcno-layer-projection}
    &u(x) = 
    W^{'l} \sigma \Bigl( W^{l} f_{\rm in}(x) + b\Bigr) + b^{'},
\end{split}
\end{equation}
where $W^{'l} \in \R^{\cdot \times d_u}$ and $b^{'} \in \R^{d_u}$.

\section{Theoretical Analysis}
\label{sec:theory}
In this section, we provide a theoretical analysis of the proposed Point Cloud Neural Operator, focusing on its invariance properties, computational complexity, and approximation capabilities. These aspects are studied through the following theorems.

Permutation invariance with respect to input points is a fundamental requirement for neural operators or networks that process point clouds, as emphasized in \cite{qi2017pointnet}. The proposed Point Cloud Neural Operator inherently satisfies this property.
\begin{theorem}[Permutation invariance]
Assume the point cloud consists of $N$ points, the Point Cloud Neural Operator $\G_{\theta}$ maps
\begin{equation}
    \Bigl[u(x^{(1)}), u(x^{(2)}), \cdots, u(x^{(N)})\Bigr] = \G_{\theta}\Bigl(\tilde{a}(x^{(1)}), \tilde{a}(x^{(2)}), \cdots, \tilde{a}(x^{(N)})\Bigr).
\end{equation}
Given any permutation $\tau: (1, 2, \cdots, N) \rightarrow  (\tau_1, \tau_2, \cdots, \tau_N)$, the Point Cloud Neural Operator $\G_{\theta}$ satisfies
\begin{equation}
    \Bigl[u(x^{(\tau_1)}), u(x^{(\tau_2)}), \cdots, u(x^{(\tau_N)})\Bigr] = \G_{\theta}\Bigl(\tilde{a}(x^{(\tau_1)}), \tilde{a}(x^{(\tau_2)}), \cdots, \tilde{a}(x^{(\tau_N)})\Bigr).
\end{equation}
\end{theorem}
\begin{proof}
 We only need to prove that the point cloud neural layer in \eqref{eq:pcno-layer} is permutation invariant. The point cloud neural layer consists of the local linear function, integral and differential operators at the discrete level (as given in \eqref{eq:integral-operator-numerics-uniform}, \eqref{eq:integral-operator-numerics}, and \eqref{eq:gradient_operator}), and the activation functions, given by the following expressions: 
 \begin{subequations}
\label{eq:permute_functions}
\begin{align}
f_{\rm out}(x^{(j)}) &:= W^{l} f_{\rm in}(x^{(j)}) + b, \\
f_{\rm out}(x^{(j)}) &:= \sum_{i=1}^{N}\kappa\bigl(x^{(j)}, x^{(i)}, f_{\rm in}\bigr)  W^{v} f_{\rm in}(x^{(i)}) \rho(x^{(i)}; \Omega) \dd\Omega_i , \\
f_{\rm out}(x^{(j)}) &:= \sum_{i: x^{(i)} \textrm{ is a neighbor of } x^{(j)}} A_i(x^{(j)})^{\dagger} \bigl(f_{\rm in}(x^{(i)}) - f_{\rm in}(x^{(j)})\bigr), \\
f_{\rm out}(x^{(j)}) &:= \sigma\bigl( f_{\rm in}(x^{(j)})  \bigr) .
\end{align}
\end{subequations}
For any permutation of the points $i \rightarrow \tau_i~(1 \leq i \leq N)$, their connectivity remains unchanged. Specifically, when $x^{(i)}$ is a neighbor of $x^{(j)}$, then $x^{(\tau_i)}$ is a neighbor of $x^{(\tau_j)}$.
Additionally, the edge weights 
$A_i(x^{(j)})^{\dagger}$, density 
$\rho(x^{(i)}; \Omega)$, and mesh size 
$ \dd\Omega_i$ will be permuted accordingly. Thus, the functions in \eqref{eq:permute_functions} remain unchanged under such a permutation, which implies that the point cloud neural layer and its composition are permutation invariant.
\end{proof}

Next, we analyze the computational complexity of the proposed Point Cloud Neural Operator. The following theorem demonstrates that the proposed operator achieves linear complexity during inference, ensuring scalability to large-scale problems.
\begin{theorem}[Linear complexity]
Assume the point cloud consists of $N$ points in $\Omega \subset \R^d$, with $E = \mathcal{O}(N)$ edges in the connectivity. The inference cost of the point cloud neural layer \eqref{eq:pcno-layer} from $f_{\rm in}: \Omega \rightarrow \R^{d_{\rm in}}$ to $f_{\rm out}: \Omega \rightarrow \R^{d_{\rm out}}$ on the point cloud, with a truncated Fourier modes number $k_{\max}$ is $\mathcal{O}\Bigl( N d_{\rm in} d_{\rm out}(k_{\max}+d) + E d_{\rm in}d\Bigr)$.
\end{theorem}
\begin{proof}
For the point cloud neural layer described in  \eqref{eq:pcno-layer}, the inference complexity of the local linear function is $\mathcal{O}\bigl(   N d_{\rm in} d_{\rm out} \bigr)$, the inference complexity of the integral at the discrete level, as given in~\eqref{eq:integral-operator-numerics-uniform} and \eqref{eq:integral-operator-numerics}, is $\mathcal{O}\bigl(N k_{\rm max} d_{\rm out} d_{\rm in}\bigr)$, 
the complexity of the differential operator at the discrete level with precomputed edge weights $A_i(x^{(j)})^{\dagger}$ is $\mathcal{O}\bigl(E d d_{\rm in} + N d_{\rm out} dd_{\rm in}\bigr)$, where the first term is from the message-passing step in~\eqref{eq:gradient_operator}. The inference complexity of the activation function is $\mathcal{O}\bigl(N d_{\rm out}\bigr)$. In total, the inference complexity is $\mathcal{O}\Bigl( N d_{\rm in} d_{\rm out}(k_{\max}+d) + E d_{\rm in}d\Bigr)$.
\end{proof}

Finally, we establish the universal approximation property, a fundamental criterion for neural network design. Our proposed Point Cloud Neural Operator satisfies this property, as a direct consequence of \cite[Theorem 2.2]{lanthaler2023nonlocal}.

\begin{theorem}[Universal approximation at continuous level]
\label{prop:uap}
Assume that $\Omega$, drawn from the data distribution $\mu$ is uniformly bounded, full-dimensional, and has a Lipschitz boundary. And both $a$ and $u$ are $L_p(\Omega)$ functions for a given $p \geq 1$.  And $\inf_{\Omega,\, x\in\Omega} \rho(x; \Omega) = \rho_{\inf} > 0$. We embed $\Omega$ into a bounded hypercube $B$ and extend the input function $\tilde{a}$ defined in \eqref{eq:tildea}, and the solution function $u$ as follows
\begin{equation}
    \tilde{a}_B(x) = 
    \begin{cases}
    \tilde{a}(x) & x \in \Omega \\
    0  & x \in B \backslash \Omega
    \end{cases}
    \quad \textrm{ and } \quad
    u_B(x) = 
    \begin{cases}
    u(x) & x \in \Omega \\
    0  & x \in B \backslash \Omega
    \end{cases},
\end{equation}
where $\tilde{a}_B(x) = [a_B(x),  x \mathds{1}_{\Omega}(x), \rho_B(x)] \in \R^{d_a + d + 1}$, the support of $\rho_B$  gives $\Omega$, and the restriction of $a_B$  on $\Omega$ recovers $a$. Consequently, both $\Omega$ and $\tilde{a}$ can be reconstructed from $\tilde{a}_B$.
We denote the extended operator as $\G^{\dagger}_B: L_p(B, \R^{d_a+d+1}) \rightarrow L_{p}(B, \R^{d_u})$, 
which satisfies $$\G^{\dagger}_B(\tilde{a}_B) = \G^{\dagger}(a,\Omega).$$
  Assume further that  $\G^{\dagger}_B$ is a continuous operator defined on the compact set $\cK_B \subset L_p(B, \R^{d_a+d+1})$, consisting of bounded functions satisfying $\sup_{\tilde{a}_B \in \cK_B}\lVert \tilde{a}_B \rVert_{L_\infty} < \infty$. 
Then, for any $\epsilon > 0$, there exists a proposed Point Cloud Neural Operator $\G_{\theta}$, such that 
\begin{equation}
\label{eq:prop:uap}
    \sup_{\tilde{a}_B \in \cK_B} \lVert \G^{\dagger}(a,\Omega) - \G_{\theta}(\tilde{a})\rVert_{L_{p}(\Omega, \R^{d_u})} \leq \epsilon.
\end{equation}
\end{theorem}
We provide a proof of this  theorem in \ref{sec:proof:uap}. The  theorem highlights that a crucial aspect of the point cloud strategy (See the literature review in \cref{ssec:literature}) to construct neural operators to handle complex and variable geometries is the introduction of the bounded hypercube $B$ and the embedding of the input and output functions within this hypercube. 
However, a key feature of our Point Cloud Neural Operator is that our implementation operates exclusively on $\Omega$, without involving $B$ or the embedding process.

\section{Numerical Results}
\label{sec:numerics}
In this section, we present numerical studies of the proposed Point Cloud Neural Operator (PCNO), focusing on its performance across a diverse set of problems involving complex and variable geometries. The numerical tests span a wide range of scenarios, from pedagogical examples in one and two dimensions to real-world applications in three dimensions:
\begin{enumerate}
    \item One-dimensional advection-diffusion problem: A pedagogical example with multiscale features,  used to investigate how the point cloud density $\rho(x; \Omega)$ affects the accuracy.
    \item Two-dimensional Darcy flow problem: A pedagogical example with a varying computational domain, examining the effects of shape changes and mesh resolution.
    \item Two-dimensional flow over an airfoil: A pedagogical example involving an airfoil, which may either be a single airfoil or an airfoil with a flap, and an adaptively generated mesh. This problem explores topology changes in the computational domain and the effects of mesh adaptation.
    \item Three-dimensional vehicle application \cite{li2024geometry}: This problem involves learning the mapping from vehicle shapes to surface pressure loads. The dataset includes vehicles from ShapeNet and Ahmed body.
    \item Three-dimensional parachute dynamics: This problem studies the mapping from parachute design shapes to their inflation behavior under pressure loads.
\end{enumerate}

All experiments involve variable computational domains, which may be discretized using a varying number of points. 
To ensure consistent and parallelizable processing across all data samples, we pad the samples (including both inputs and outputs) with zeros so that all data samples reach the prescribed maximum length. A node mask vector is then used to distinguish between the original points and the padded ones. The implementation is designed such that garbage values on the padded points do not affect the original points.    
For the network architecture, PCNO employs four point cloud neural layers~in \eqref{eq:pcno-layer}, each with a channel size of $128$. The widths of both the lifting and projection layers are also set to $128$, and the $\texttt{GeLU}$ activation function is adopted. The truncated Fourier frequency $k$ is fixed across all integral operators but may vary between experiments due to memory limitations. We observe that increasing the truncated Fourier frequency generally reduces test errors.
Training is performed using the Adam optimizer~\cite{kingma2015adam} with default hyperparameters ($\beta_1 = 0.9$, $\beta_2 = 0.999$) and a weight decay factor of $10^{-4}$ over $500$ epochs.  
The base learning rate, denoted as $\gamma$, is  tuned and varies across different experiments. The length scale vector $L$ in the integral operator~\eqref{eq:pcno-layer} is treated as a hyperparameter, and its corresponding base learning rate may differ from other parameters. 
We utilize the OneCycleLR scheduler~\cite{smith2019super} to dynamically adjust the learning rate during training. Specifically, the learning rate starts at $\frac{\gamma}{2}$, increases linearly to the base learning rate $\gamma$ over the first $20\%$ of total epochs, and then decreases following a cosine annealing schedule to $\frac{\gamma}{100}$. The relative $L_2$ error is used as the cost function for both training and testing. All training and testing are conducted on a single NVIDIA A100 80G GPU.
A comprehensive comparison with other state-of-the-art neural operators is provided in \ref{sec:comparison}, focusing on widely used benchmarks, but restricted to fixed unit cube computational domains.
Our code and datasets are publicly available at: \url{https://github.com/PKU-CMEGroup/NeuralOperator}.

\subsection{Advection Diffusion Problem}
\label{ssec:adv}
In this subsection, we consider solving the steady state, one-dimensional advection diffusion boundary value problem
\begin{equation*} 
\begin{split}
\frac{\partial u}{\partial x} - D \frac{\partial^2 u}{\partial^2 x} = f(x) \quad \textrm{in} \quad \Omega, \\
u(0) = u_l, u(L) = 0,
\end{split}
\end{equation*} 
defined on a variable computational domain $\Omega = [0, L]$. Here 
$f(x)$ represents the source field, $D$ is the constant diffusivity, and $u$ denotes the steady state solution. The left boundary condition is $u_l$, while the right boundary condition is fixed at $0$, resulting in a steep boundary layer at the right end (See \cref{fig:adv_pred}).
The computational domain length $L$ is uniformly sampled from $U[10,15]$, the left boundary condition $u_l$ is uniformly sampled from $U[0, 1]$, the diffusivity is uniformly sampled from $U[5\times10^{-3}, 5\times10^{-2}]$. The source function is a Gaussian random field defined as
\begin{equation*}
    f(x) = \begin{cases}
        |g(L_f x)|  &  0 \leq x \leq L_f , \\
        0           &  x > L_f ,
\end{cases}
\end{equation*}
where the support length $L_f$ is uniformly sampled from $U[5,8]$.
The field $g(x)\sim \mathcal{N}(0, 625(-\Delta + 5^2)^{-2})$ is a Gaussian random field with zero boundary condition.

We generate the dataset by solving the boundary value problem using the finite difference method with three types of meshes. 
The first is a uniform mesh with a mesh size of $2\times10^{-3}$ (Uniform). 
The second is an exponentially graded mesh, with the mesh size set to $10^{-4}$ at the right end and $10^{-2}$ at the left end, employing a growth ratio of $1.05$ (Exponential). 
The third is a linearly graded mesh with mesh size set to $10^{-3}$ at the right end and $10^{-2}$ at the left end (Linear). 
The latter two meshes are adaptive to better resolve the boundary layer near the right end (See \cref{fig:adv_pred}).
For each mesh type, we randomly generate 2500 samples, which are then combined into a mixed mesh dataset (Mixed). The operator maps the source function $f$, diffusivity $D$, the left boundary condition $u_l$ (treated as constant functions over $\Omega$), and the domain $\Omega$ to the corresponding solution $u$:
\begin{equation*}
    \G^{\dagger} : (f, D, u_l, \Omega) \mapsto u. 
\end{equation*}

\begin{figure}
    \centering
    \includegraphics[width=0.9\linewidth]{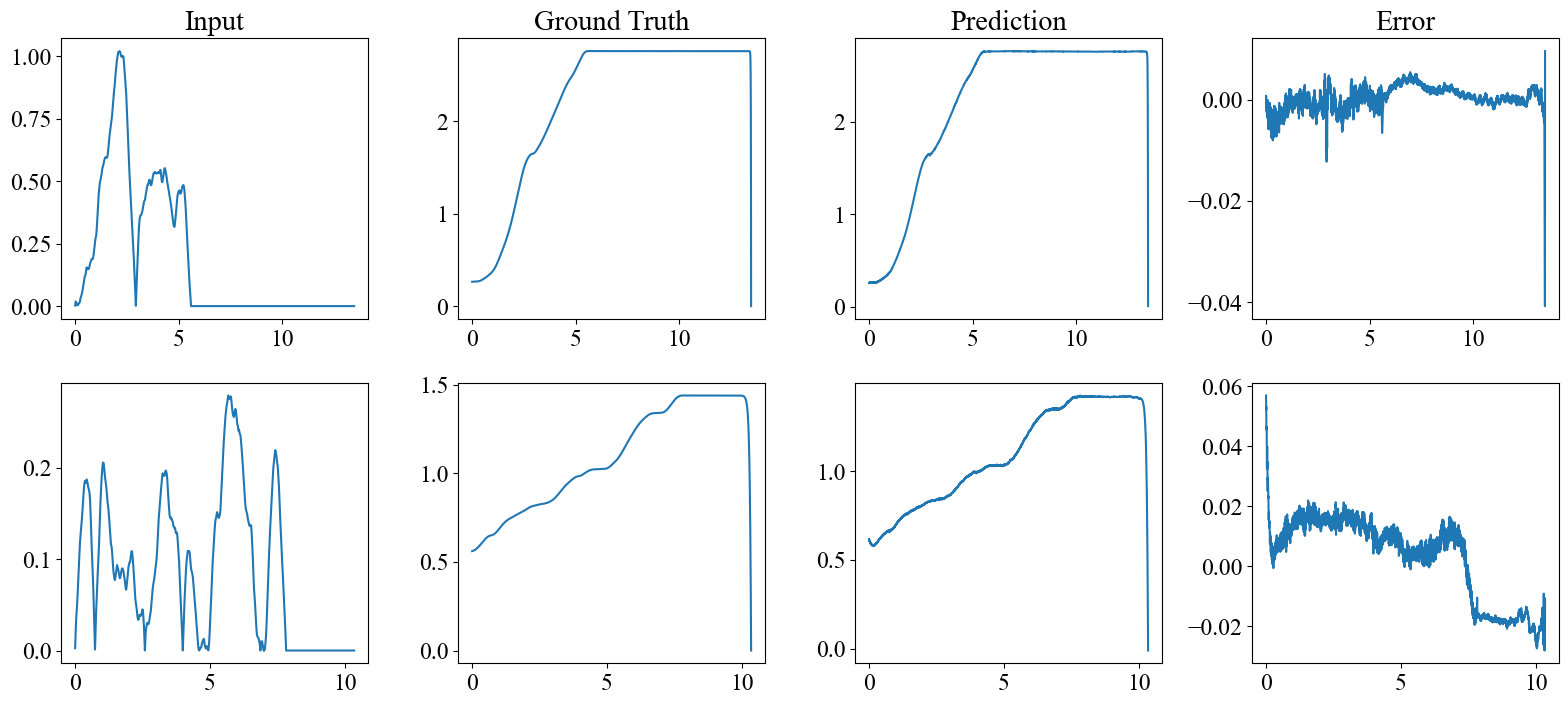}
    \caption{Advection diffusion problem: From left to right, the panels display the source field $f(x)$, the reference solution $u(x)$, the predicted solution, and the error (defined as the difference between the reference and predicted solutions), for the test cases with the median relative $L_2$ error (top) and the largest relative $L_2$ error (bottom).}
    \label{fig:adv_pred}
\end{figure}

Next, we discuss the training process for the PCNO model. We use 64 Fourier modes in the sole dimension and set a uniform density $\rho(x,\Omega) = \frac{1}{|\Omega|}$ in the integral operator~ \eqref{eq:integral-operator-rho}.
We first train the PCNO with 1000 data points from the Mixed dataset and test it with 600 data points. In this setup, our PCNO model achieves a relative $L_2$ test error of $0.167\%$ . The relative training and test errors during iterations, as well as the distribution of test errors for each mesh type, are visualized in \cref{fig:adv_loss}. 
The error distributions across the three mesh types are similar, indicating that the PCNO model can effectively handle meshes with different point distributions.
\Cref{fig:adv_pred} further illustrates the model's predictions, showing test samples with median and largest relative $L_2$ errors, both of which belong to the uniform distribution mesh. The largest error can be attributed to the pronounced fluctuations in the source field, which result in reduced smoothness or regularity in the solution. However, even in the case of the largest error, the overall error remains very small.

\begin{figure}
    \centering
    \includegraphics[width=0.39\linewidth]{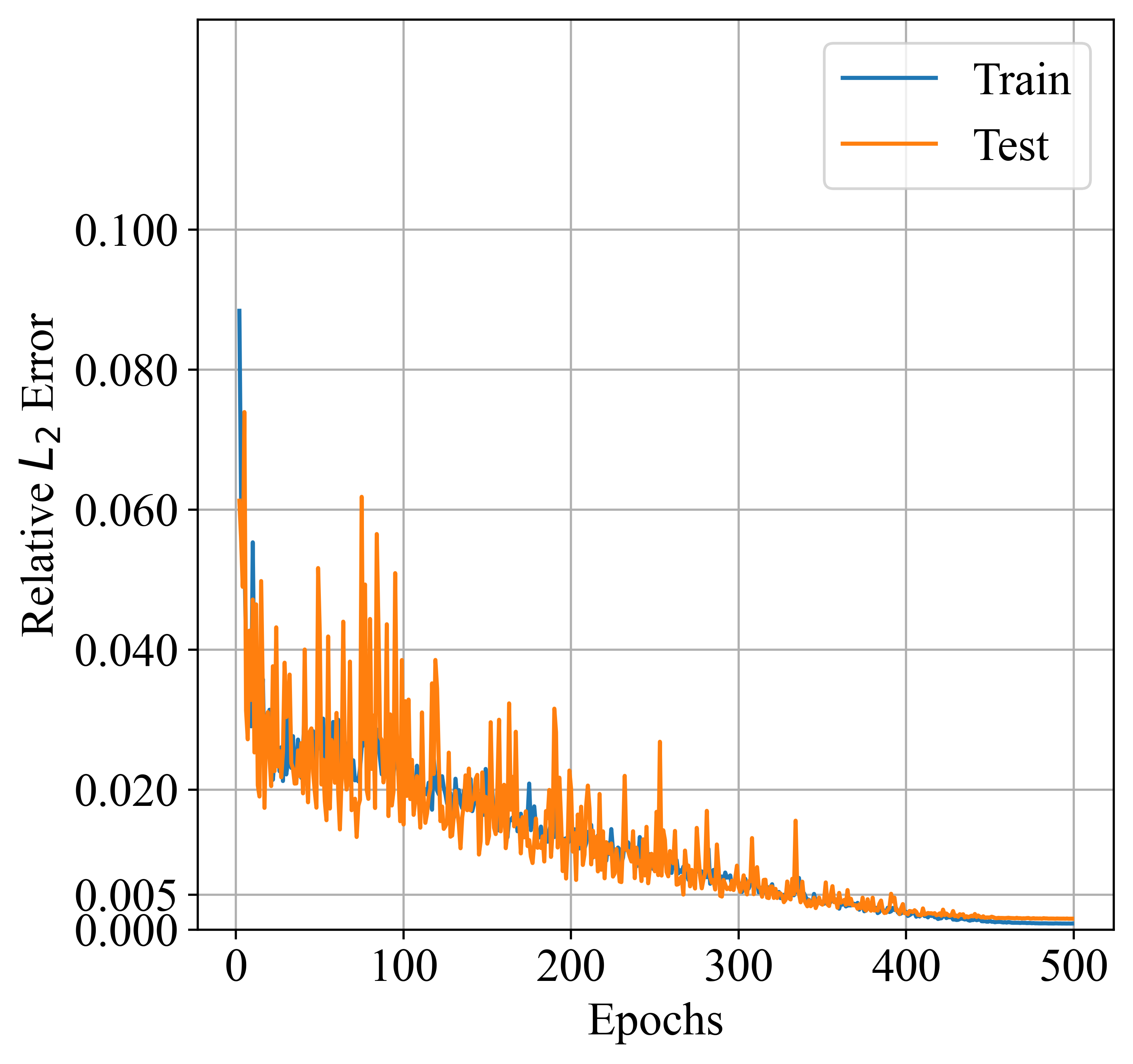}~~~~
    \includegraphics[width=0.355\linewidth]{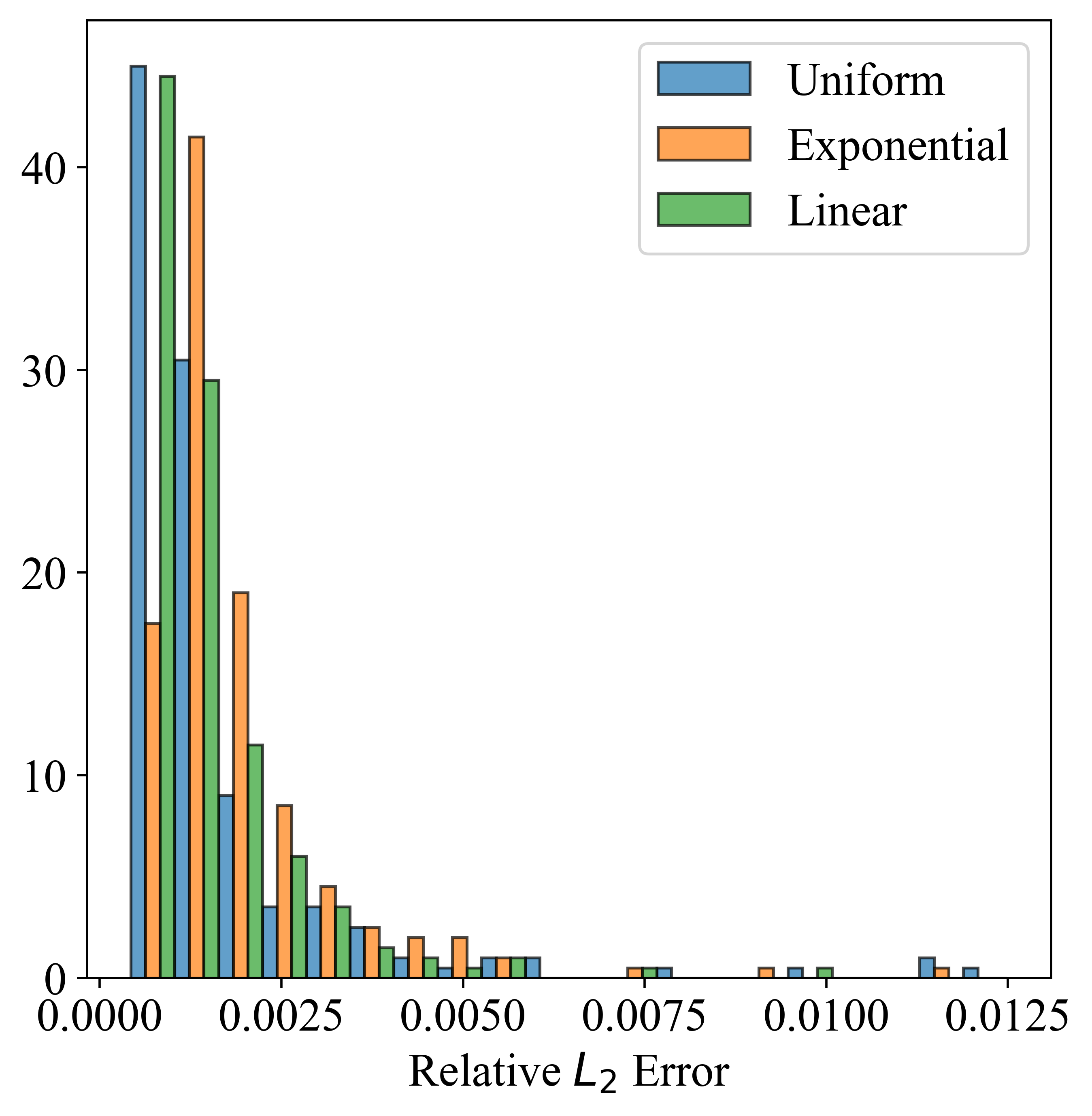}
    \caption{Advection diffusion problem: The relative training and test errors over epochs (left) and the distribution of the test errors (right). The training and test datasets consist of 1000 and 600 samples, respectively, uniformly sampled from the Mixed dataset. The test error distribution is visualized separately for the uniform, exponential, and linear mesh cases.}
    \label{fig:adv_loss}
\end{figure}

Furthermore, we study the choice of density function $\rho(x,\Omega)$ introduced in \cref{ssec:integral-op}. Specifically, we train and test the PCNO model using both uniform density and point cloud density on the Mixed dataset with 500, 1000, and 1500 training samples. The test errors are shown in \cref{fig:adv_equal_weight} (left). 
Although point cloud density incorporates mesh adaptivity information, which tends to improve accuracy, the test results in \cref{fig:adv_equal_weight} show that the point cloud density performs notably worse compared to uniform density.
This disparity arises because there are significant variations in point cloud density across these three types of meshes. However, the use of point cloud density requires  
that each $\Omega$ be associated with a unique $\rho(x;\Omega)$. Violating this requirement causes point cloud density to fail in accurately capturing the necessary information.

Finally, we point out that neural operators trained on one type of discretization (e.g., Uniform, Exponential, or Linear) may generalize well to other types of discretization. We train the PCNO models on these four different datasets: Uniform, Exponential, Linear, and Mixed, with 500, 1000, and 1500 training samples, and test them on the Mixed dataset with 600 data points. 
The results are shown in \cref{fig:adv_equal_weight} (right). The figure shows except for the Uniform dataset (which may suffer from overfitting), training on the various datasets, results in relative test errors scaling at a rate of $\bigO(n^{-1/2})$ as the number of training samples increases. This highlights the importance of designing neural operators at the continuous level to mitigate the impact of discretization.
The figure further demonstrates that trained on the Mixed dataset yields the best performance, guiding us to train neural operators on mixed discretization datasets to enhance generalization and robustness. 

\begin{figure}
    \centering
    \includegraphics[width=0.4\linewidth]{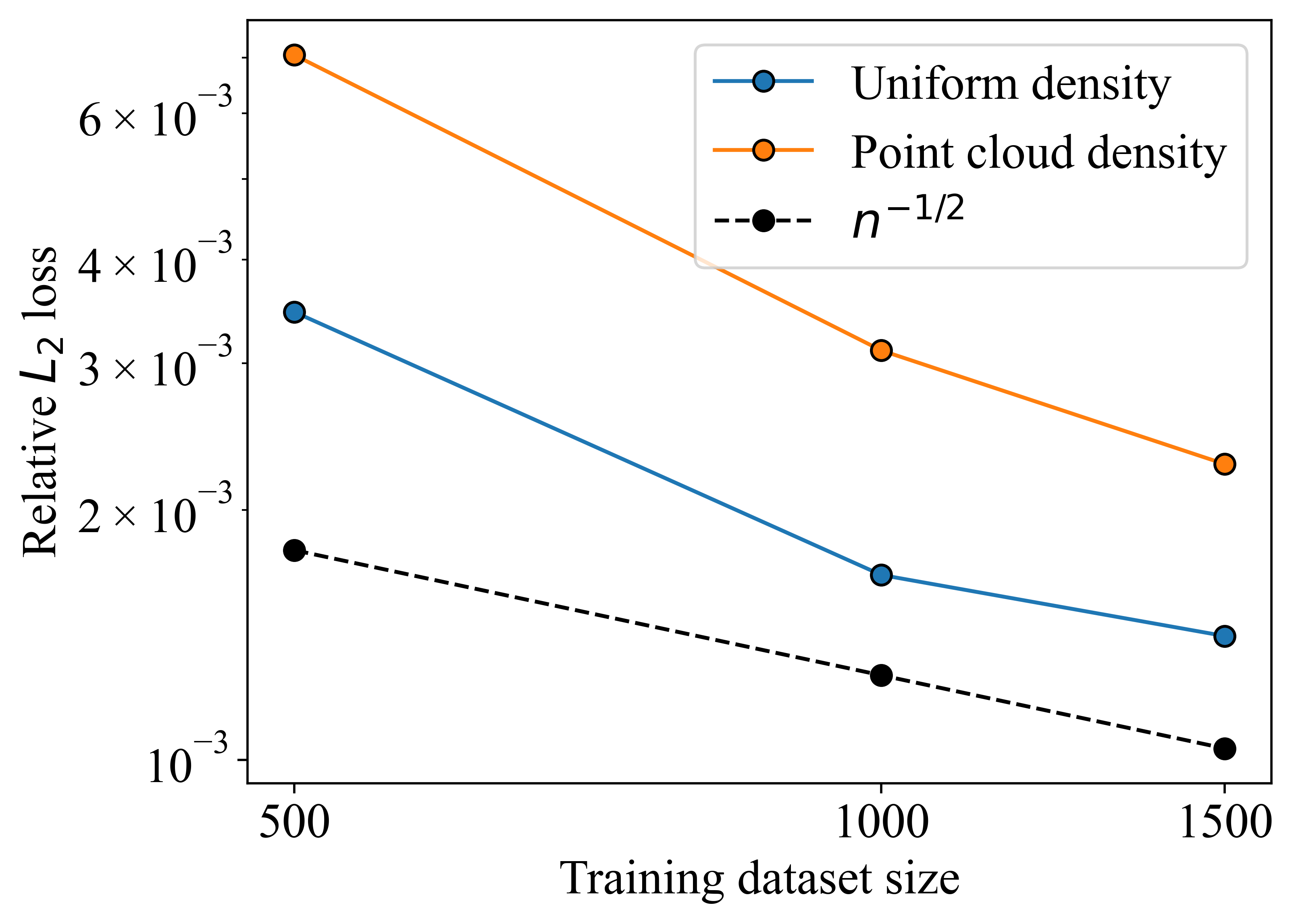}
    \includegraphics[width=0.55\linewidth]{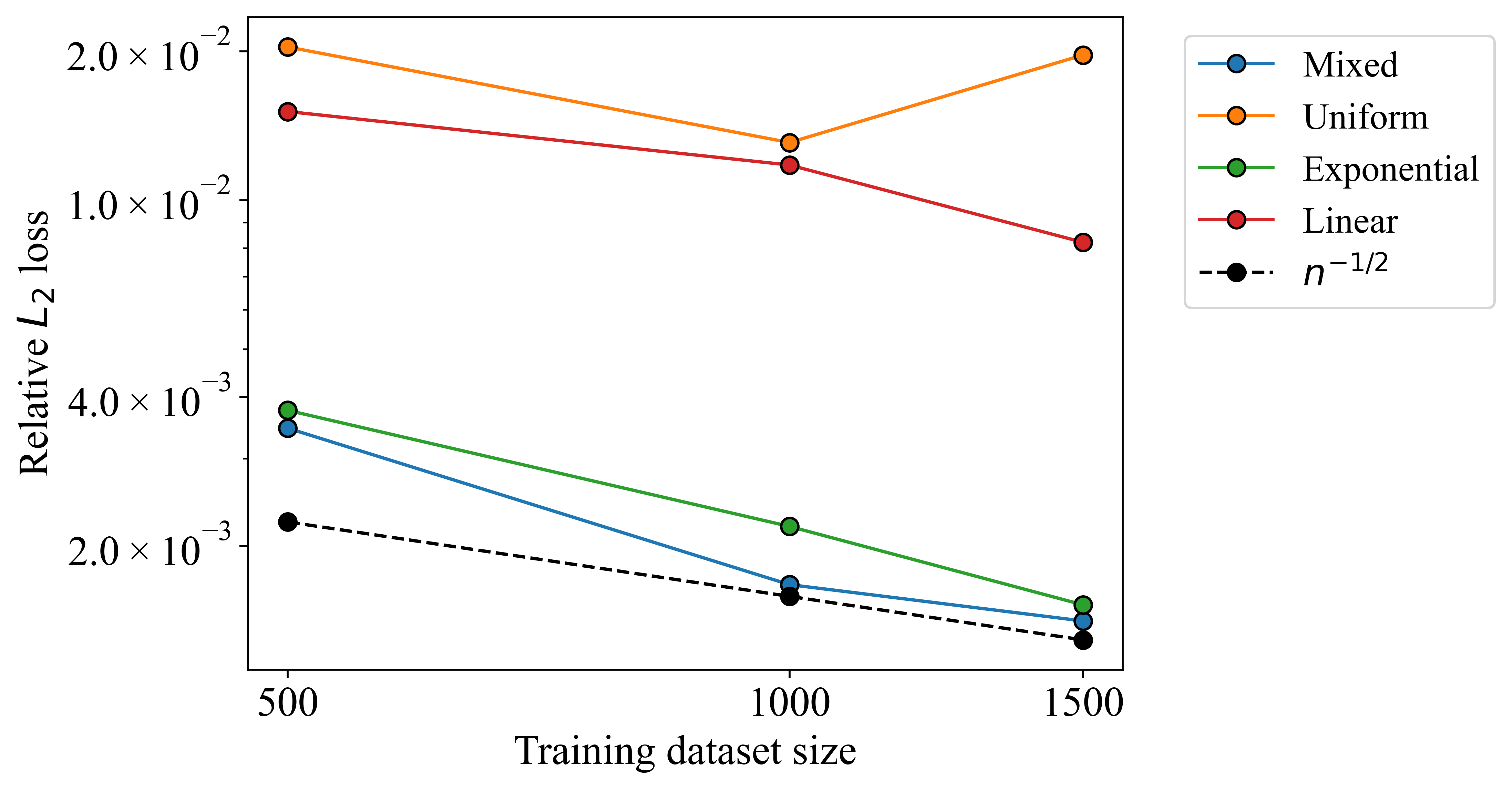}
    \caption{Advection diffusion problem: The relative test errors of PCNO with uniform density and point cloud density for $\rho(x,\Omega)$ trained on Mixed datasets of varying sizes (left). The relative test errors of PCNO trained on different datasets (Mixed, Uniform, Exponential, and Linear) with varying dataset sizes and tested on the Mixed dataset (right).}
    \label{fig:adv_equal_weight}
\end{figure}

\subsection{Darcy Flow Problem}
\label{ssec:darcy}

In this subsection, we consider solving the steady-state, two-dimensional Darcy
flow equation
\begin{equation*} 
\begin{split}
-\nabla\cdot(a\nabla u) &= f~\text{in}~~\Omega, \\
u &= 0~\text{on}~\partial \Omega,
\end{split}
\end{equation*} 
defined on a variable computational domain $\Omega$. 
Here $f = 1$ represents the fixed source field, $a$ is the positive permeability field, and $u$ denotes the pressure field. 
To parametrize the boundary $\partial \Omega$ of the 
computational domain $\Omega$, we use polar coordinates given by  
$$(x, y) = (r(\theta)\cos(2\pi\theta), r(\theta)\sin(2\pi\theta)),\quad \theta \in [0,1],$$ 
where $r(\theta)$ is a function constructed by connecting randomly sampled points 
$\{(\frac{j}{n},r_j)\}_{j=0}^n$ with $r_0 = r_n$ and $n= 5$ with Gaussian process regression with radial basis function (RBF) kernel.  The values $r_j, 0 \leq j \leq n-1$ are uniformly sampled from the interval $U [0.5,1.5]$. Note that there might be a kink at $\theta=0$, resulting in a non-smooth boundary for the domain.
The permeability field is generated by first sampling $a(ih,jh)$ for $h = 0.75$ and $-2 \leq i,j \leq 2$ uniformly from  $U[0.5,1.5]$, then connecting these randomly sampled points with Gaussian process regression.

We generate the dataset by solving the Darcy flow equation using the finite element method at two mesh resolutions generated with Gmsh~\cite{geuzaine2009gmsh}. The fine mesh has a resolution of approximately $0.03$, while the coarse mesh has a resolution of approximately $0.06$.  For each mesh resolution, we generate 2000 data pairs. The operator then maps from the permeability field $a$ and the domain $\Omega$ to the corresponding solution $u$:
\begin{equation*}
    \G^{\dagger} : (a, \Omega) \mapsto u. 
\end{equation*}

Next, we discuss the training process for the PCNO model.
We retain 16 Fourier modes in each spatial direction, summing over $k = (k_1, k_2)$ where $-16 \leq k_i \leq 16$ in \eqref{eq:pcno-layer}, and set a uniform density $\rho(x;\Omega)=\frac{1}{|\Omega|}.$
The PCNO model is trained on a dataset consisting of 500 fine-mesh samples and 500 coarse-mesh samples, resulting in a total of 1000 training samples. The model is then tested on a separate dataset comprising 200 fine-mesh samples and 200 coarse-mesh samples.  
In this setup, the PCNO model achieves a relative $L_2$ test error of approximately $0.683\%$. The relative training and test errors over iterations, as well as the distribution of test errors, are visualized in \cref{fig:darcy_flow_loss}. The test errors and their distributions for fine-mesh and coarse-mesh samples are comparable, with all test errors being notably small. 
\Cref{fig:darcy_flow_pred} further visualizes the model's predictions, showing test samples with median and largest relative $L_2$ errors, the largest error and median error both occur on the coarse mesh.
Even in the case of the largest error, the PCNO model successfully identifies the boundaries of the point cloud where the solution approaches zero. Meanwhile, the error distribution is uniform, without the phenomenon of suddenly large errors occurring near some points, demonstrating its capability to handle complex geometries effectively.

\begin{figure}
    \centering
    \includegraphics[width=0.39\linewidth]{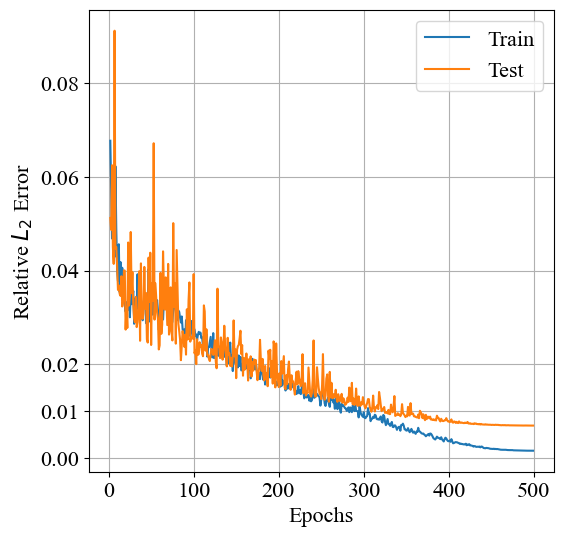}~~~~
    \includegraphics[width=0.355\linewidth]{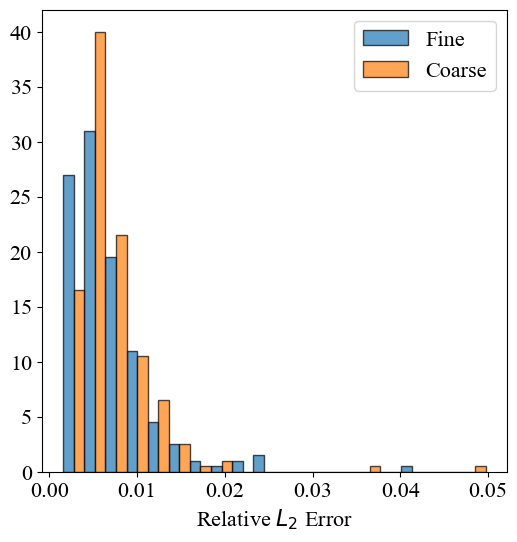}
    \caption{
Darcy flow problem: The relative training and test errors over epochs (left) and the distribution of the test errors (right). The training and test datasets consist of 1000 and 400 samples, respectively, equally sampled for fine-mesh and coarse-mesh cases. The test error distribution is visualized separately for fine-mesh (Fine) and coarse-mesh (Coarse) cases.
}
    \label{fig:darcy_flow_loss}
\end{figure}

\begin{figure}
    \centering
    \includegraphics[width=0.9\linewidth]{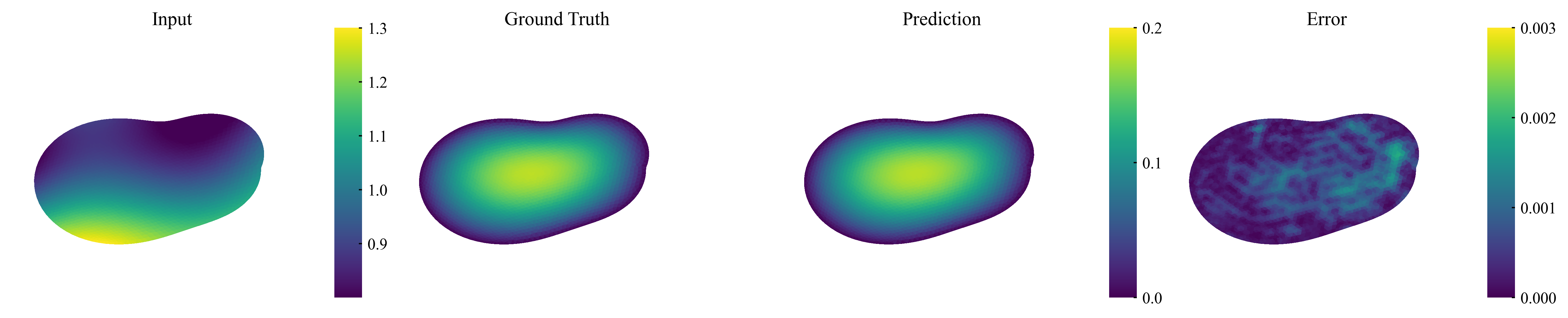}
    \includegraphics[width=0.9\linewidth]{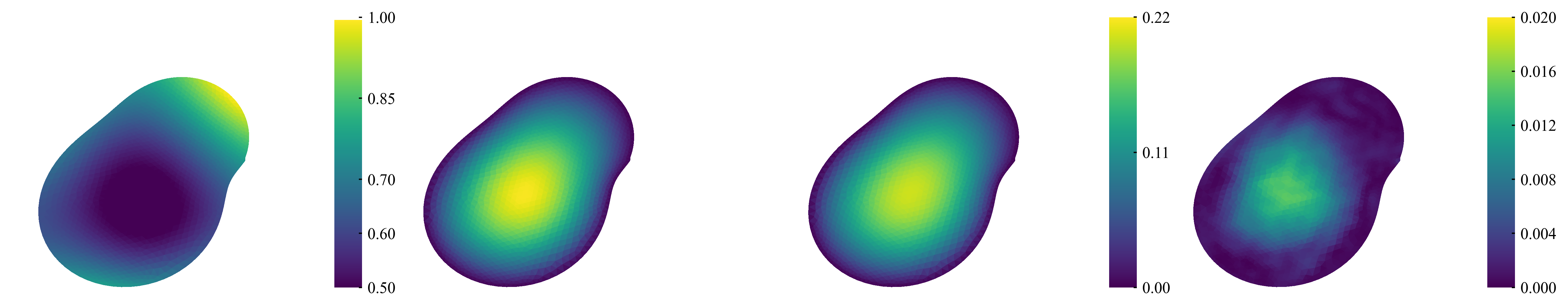}
    \caption{Darcy flow problem:  
    From left to right, the panels display the permeability field $a(x)$, the reference solution $u(x)$, the predicted solution, and the error (defined as the difference between the reference and predicted solutions), for the test cases with the median relative $L_2$ error (top) and the largest relative $L_2$ error (bottom).}
    \label{fig:darcy_flow_pred}
\end{figure}

Furthermore, we evaluate the performance of PCNO with respect to mesh resolution.  
Specifically, we consider three types of training configurations using different datasets: fine-mesh data (Fine), coarse-mesh data (Coarse), and an equal-size mixture of both (Mixed). The training datasets contain 500, 1000, and 1500 samples. 
For the test datasets, we use 200 fine-mesh samples (Fine) and 200 coarse-data samples (Coarse). Models trained under these different configurations are evaluated on both test datasets, and the results are presented in \cref{fig:darcy_flow_study}.
The results show that for all configurations, as the number of training samples increases, the relative test error decreases at a rate of $\bigO(n^{-1/2})$. This trend persists even when the training and test datasets have different resolutions, which highlights the robustness of neural operators formulated at the continuous level in mitigating the effects of discretization.
Training PCNO on the Mixed dataset yields superior performance.
Notably, combining 750 fine-mesh data with 750 coarse-mesh data outperforms training exclusively on 750 or even 1500 fine-mesh samples.  This result emphasizes the potential of leveraging multifidelity datasets to improve the training of neural operators.

\begin{figure}
    \centering
    \includegraphics[width=0.45\linewidth]{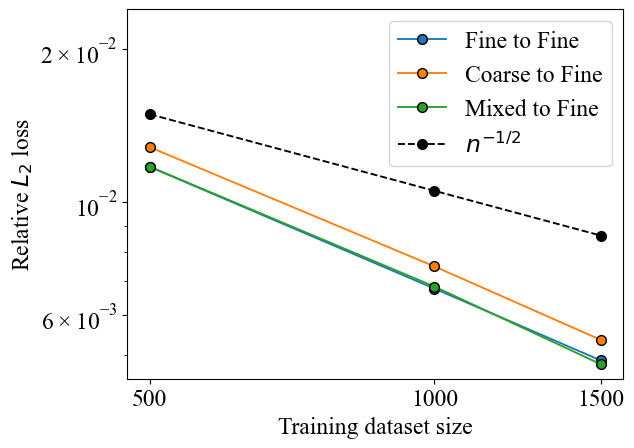}
    \includegraphics[width=0.45\linewidth]{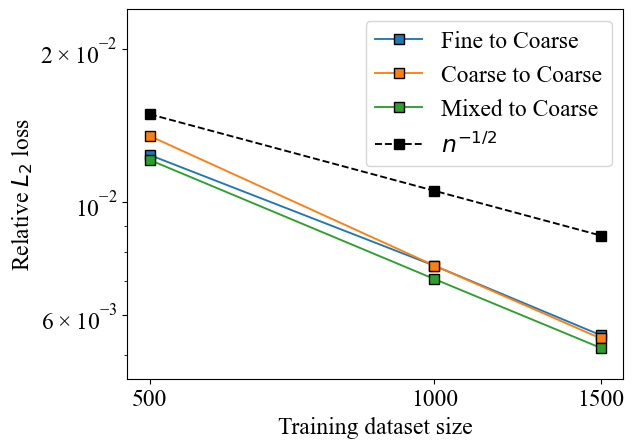}
    \caption{Darcy flow problem: The relative test errors of PCNO trained on different datasets (Fine, Coarse, and Mixed) with varying dataset sizes,  and tested on the fine-mesh dataset (left) and the coarse-mesh dataset (right).}
    \label{fig:darcy_flow_study}
\end{figure}

\subsection{Flow Over Airfoil}
\label{ssec:airfoil_flap}
In this subsection, we consider the flow over airfoil problem by solving the steady-state, two-dimensional Euler equations
\begin{equation}
\begin{split}
\nabla \cdot (\rho \bm{v}) = 0, \\
\nabla \cdot (\rho \bm{v} \otimes \bm{v} + p \I) = 0, \\
\nabla \cdot  \Bigl( (E + p)\bm{v} \Bigr) = 0 
\end{split}
\end{equation}
in a variable domain $\Omega$ surrounding different airfoil configurations. Here $\rho$ represents the fluid density, $\bm{v}$ is the velocity vector, $p$ is the pressure, and $E$ is the total energy. The viscous effect is ignored. The far-field boundary conditions are specified as 
$$\rho_{\infty} = 1 , p_{\infty}  = 1.0 , M_{\infty} = 0.8,$$
where $M_{\infty}$ is the Mach number, and hence shock waves are generated around the airfoil.  At the airfoil, a no-penetration velocity condition $\bm{v} \cdot n = 0$ is imposed.
The computational domain is defined as a circle with a radius of $r = 50$, excluding the airfoil shape. Two airfoil configurations are considered: a single main airfoil (Airfoil) and a main airfoil with a flap (Airfoil+Flap). Consequently, the computational domain exhibits varying topologies. Both the main airfoil and the flap are generated from NACA four-digit airfoil profiles. The main airfoil has a chord length of $1$, while the flap has a chord length of $0.2$ and is positioned at a distance of $(-0.015, 0.05)$ relative to the main airfoil. Other design parameters, such as camber, thickness, and angle of attack, are varied, with details provided in \cref{tab:airfoil}.

\begin{table}
\centering
\begin{tabular}{|c|c|c|}
\hline 
\multicolumn{2}{|c|}{Design variable} & \multicolumn{1}{c|}{Range} \\
\hline 
\multirow{4}{*}{Main airfoil} &  Camber-to-chord ratio   & $0\% \sim 9\%$     \\   
&  Max camber location    & $20\% \sim 60\%$       \\  
&  Thickness-to-chord ratio  & $5\% \sim 30\%$     \\  
&  Angle of attack  & $-5^\circ \sim 20^\circ$   \\  
\cline{1-3}
\multirow{4}{*}{Flap} &   Camber-to-chord ratio    & $0\% \sim 9\%$     \\   
&  Max camber location   & $20\% \sim 60\%$       \\  
&  Thickness-to-chord ratio & $10\% \sim 20\%$     \\  
&  Relative angle of attack & $5^\circ \sim 40^\circ$   \\  
\hline 
\end{tabular}
\caption{Flow over airfoil: The geometric parameters include the camber-to-chord ratio, the location of maximum camber relative to the chord length, the thickness-to-chord ratio, and the angle of attack for both the single airfoil configuration and airfoil-with-flap configuration.}
\label{tab:airfoil}
\end{table}

We generate the dataset by solving the Euler equations using the AERO-Suite \cite{farhat1998load,wang2011algorithms,farhat2010robust,huang2018family,borker2019mesh,michopoulos2024bottom}. The mesh is generated adaptively by Gmsh~\cite{geuzaine2009gmsh}, with sizes near the main airfoil and flap approximately $5 \times 10^{-3}$ and $2.5 \times 10^{-3}$, respectively, and $5$ at the far field. Within the domain, the mesh size transitions linearly. For each configuration, we generate approximately 2000 data. The operator maps the domain $\Omega$ to the pressure field $p$:
\begin{equation*}
    \G^{\dagger} : \Omega \mapsto p. 
\end{equation*}

Next, we discuss the training process for the PCNO model. We retain 16 Fourier modes in each spatial direction and set the density function as the point cloud density: $\rho(x^{(i)};\Omega) = \frac{1}{N |\dd \Omega_i|}$ in \eqref{eq:integral-operator-rho}. 
It is worth mentioning this choice of density function yields superior results compared to using a uniform density function.
The Point Cloud Neural Operator is first trained on a mixed dataset consisting of 500 samples from the single-airfoil configuration and 500 samples from the airfoil-with-flap configuration, resulting in a total of 1000 training samples. The network is then tested on 500 mixed samples. In this setup, the PCNO model achieves a relative $L_2$ test error of approximately  $1.83\%$, compared to $8.09\%$ test error when using the uniform density function. 
The relative training and test errors over iterations, as well as the distribution of test errors, are visualized in \cref{fig:airfoil_loss}.
The error distributions across the two configurations are similar, indicating that the PCNO model can effectively handle topology variations.
\Cref{fig:airfoil_pred} further illustrates the model's predictions, showing test samples with median and largest relative $L_2$ errors. It is observed that the error is concentrated along the shock, with the case of the largest error exhibiting a deviation in identifying the shock location.
\begin{figure}
    \centering
    \includegraphics[width=0.39\linewidth]{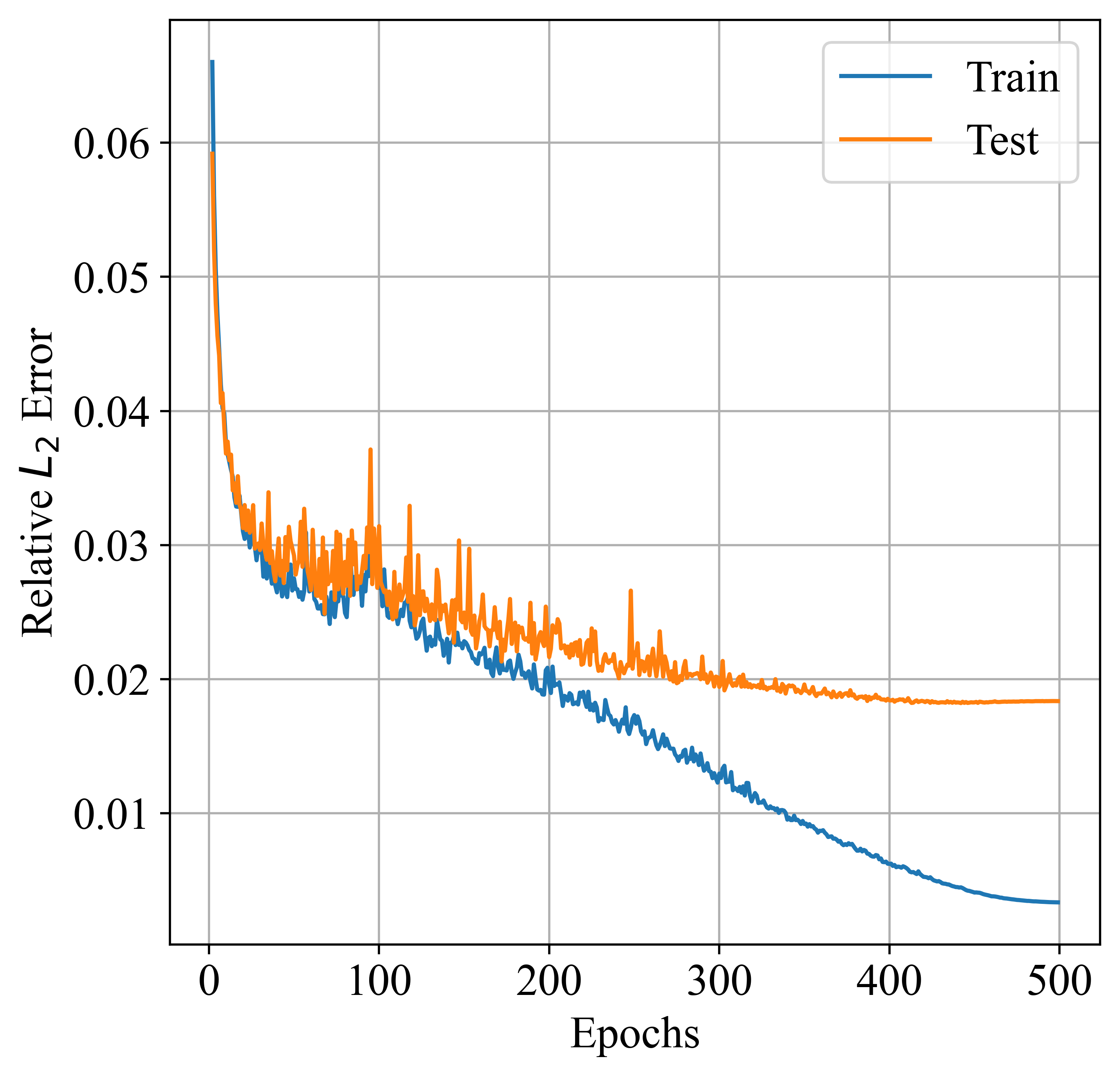}~~~~
    \includegraphics[width=0.355\linewidth]{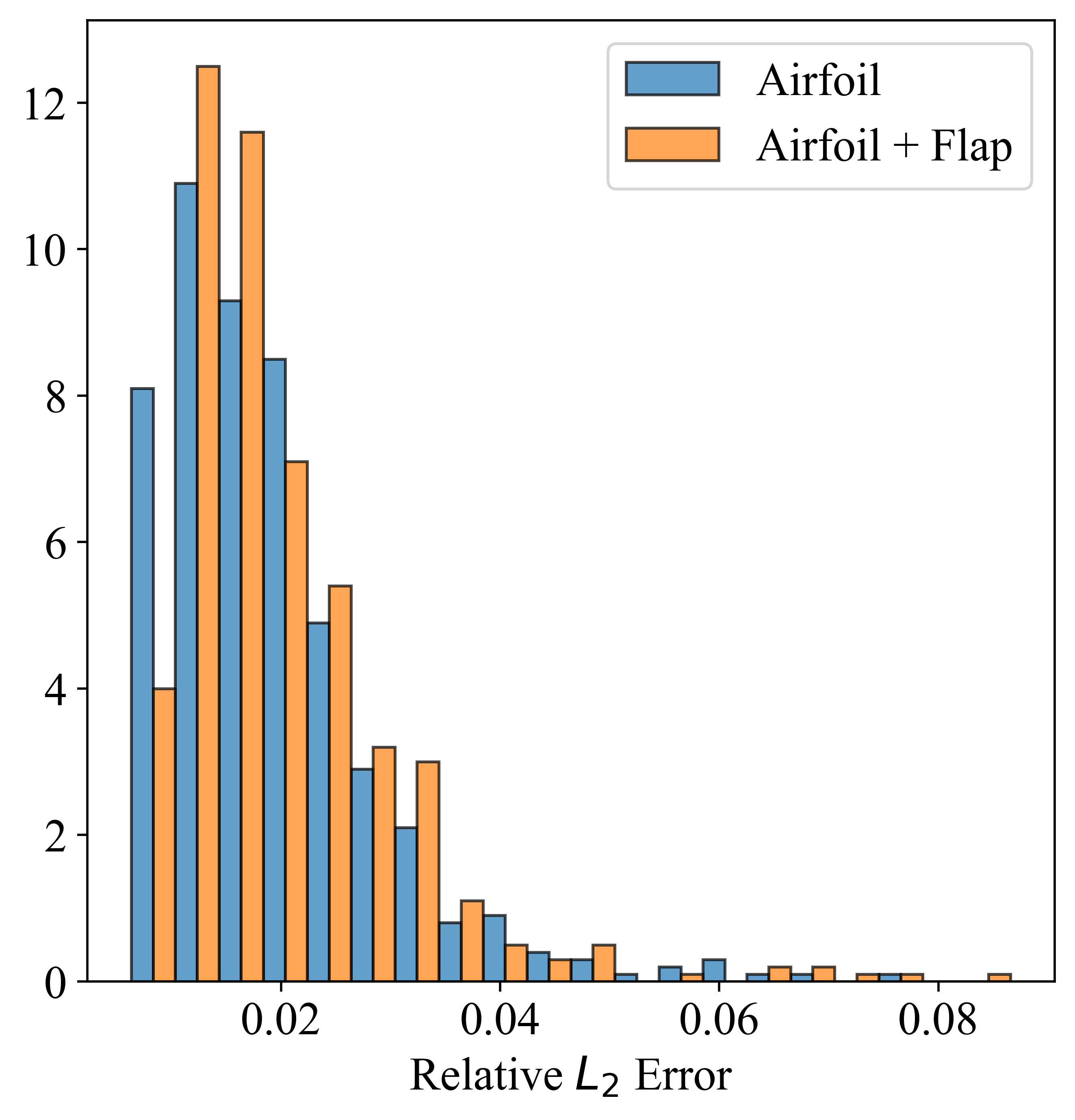}
    \caption{
    Flow over airfoil: The relative training and test errors over epochs (left) and the distribution of the test errors (right). The training and test datasets consist of 1000 and 500 samples, respectively, equally sampled for single-airfoil (Airfoil) and airfoil-with-flap (Airfoil+Flap) configurations. The test error distribution is visualized separately for each configuration.}
    \label{fig:airfoil_loss}
\end{figure}

\begin{figure}
    \centering
    \includegraphics[width=\linewidth]{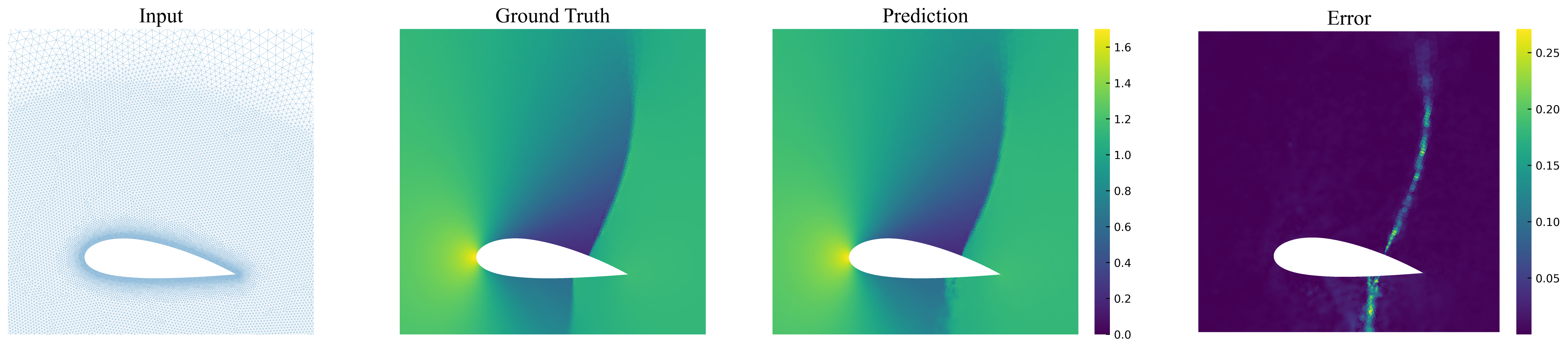}
    \includegraphics[width=\linewidth]{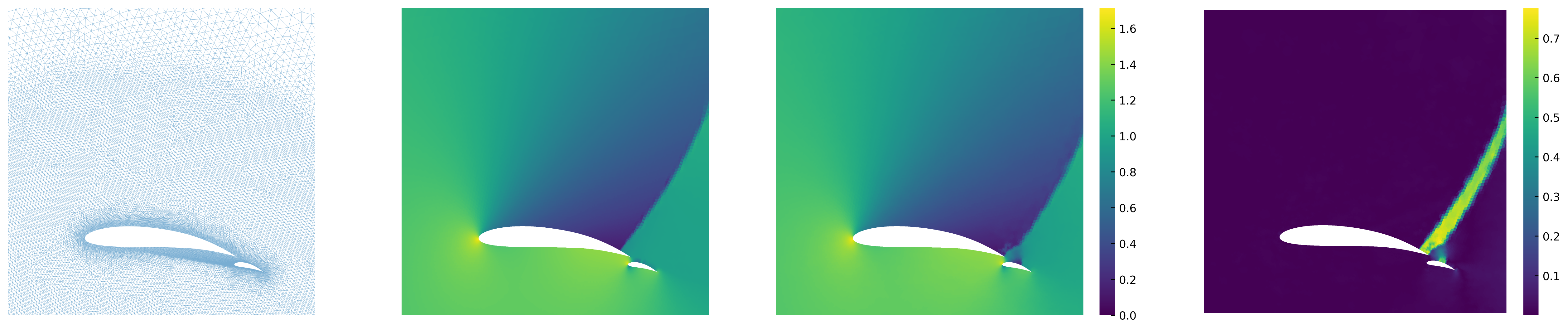}
    \caption{
    Flow over airfoil:  
    From left to right, the panels display the input mesh representing $\Omega$, the reference solution $p(x)$, the predicted solution, and the error (defined as the difference between the reference and predicted solutions), for the test cases with the median relative $L_2$ error (top) and the largest relative $L_2$ error (bottom).
    }
    \label{fig:airfoil_pred}
\end{figure}

Finally, we evaluate the performance of the PCNO with respect to topology variations. Specifically, we investigate three types of training configurations: the single-airfoil configuration (Airfoil), the airfoil-with-flap configuration (Airfoil+Flap), and an equal-size mixture of both (Mixed). For each configuration, we train the PCNO with datasets containing 500, 1000, and 1500 samples. We then evaluate the performance by testing on both single-airfoil and airfoil-with-flap configurations. The relative $L_2$ errors are presented in \cref{fig:airfoil_loss_compare}. For this problem, the error decreases with the training data size at a rate of approximately $\bigO(n^{-1/5})$ instead of $\bigO(n^{-1/2})$. 
A possible explanation is that the predicted pressure field is rough and features discontinuities, with the error being concentrated near these discontinuities.
Furthermore, the top flat curves in \cref{fig:airfoil_loss_compare} indicate that training the PCNO on one configuration does not generalize well to the other. Training on the Mixed dataset with 1000 data performs similarly to training on a single configuration (Airfoil or Airfoil+Flap) with 500 data and testing on the same configuration. This indicates that although PCNO does not show generalization across topology variations, the model can effectively handle topology variations by adapting to different configurations within the same framework.


\begin{figure}
    \centering
    \includegraphics[width=0.45\linewidth]{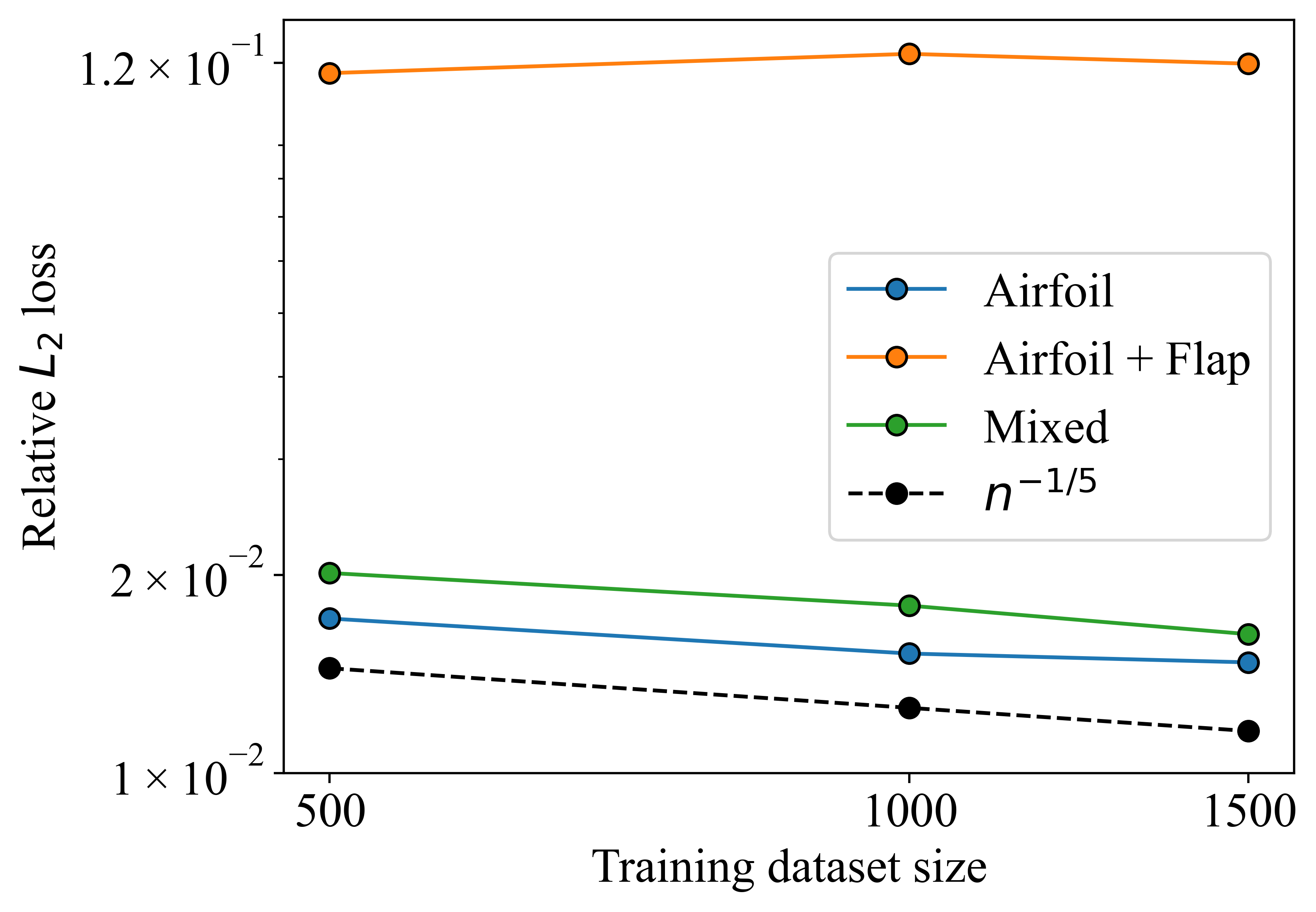}
    \includegraphics[width=0.45\linewidth]{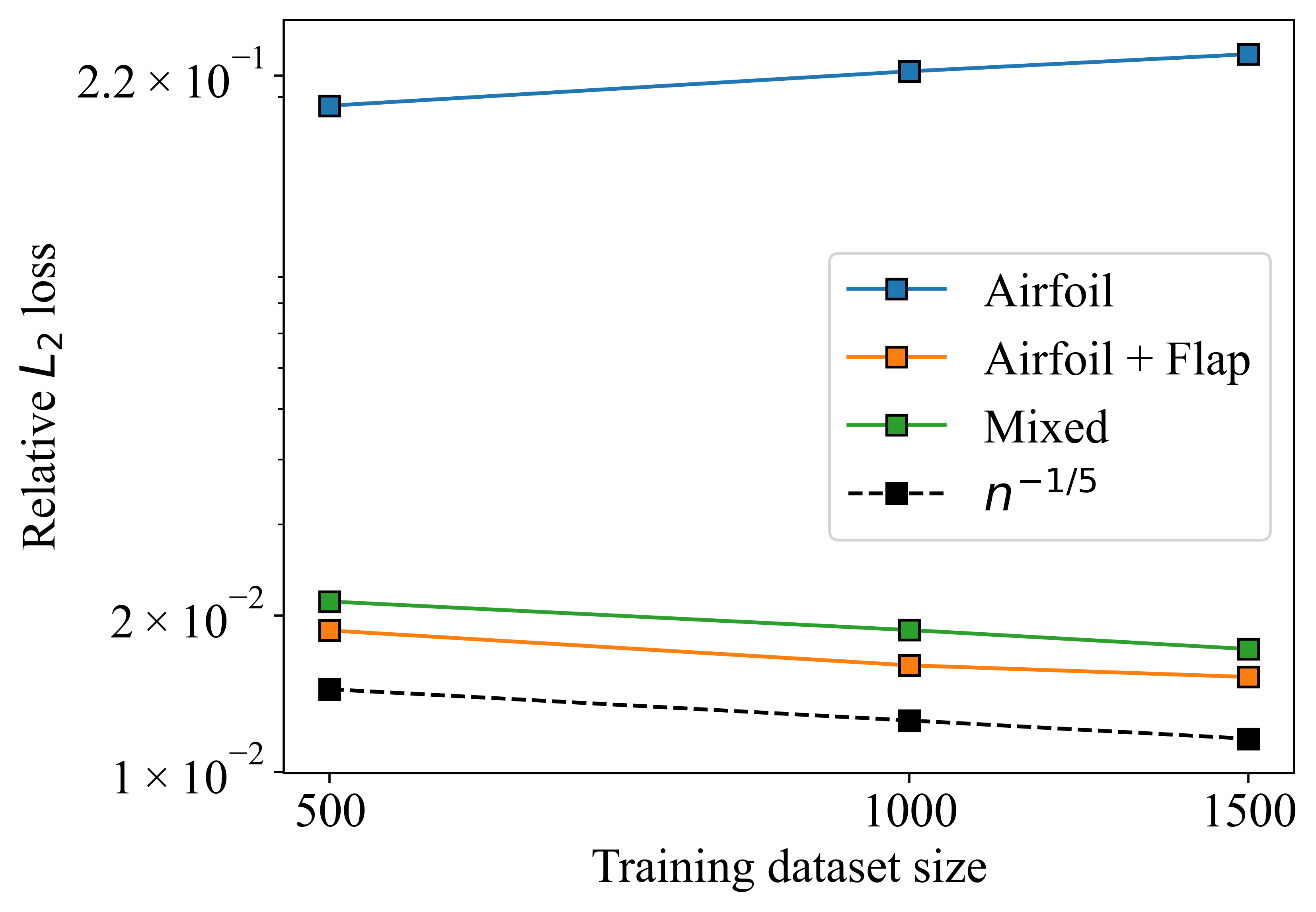}
    \caption{Flow over airfoil:  The relative test errors of PCNO trained on different configurations,  single-airfoil (Airfoil), airfoil-with-flap (Airfoil+Flap), and an equal-size mixture of both (Mixed), with varying dataset sizes,  and tested on both the single-airfoil configuration (left) and the airfoil-with-flap configuration (right).}
    \label{fig:airfoil_loss_compare}
\end{figure}


\subsection{Vehicle Application}
\label{ssec:vehicle}

In this subsection, we consider predicting the surface pressure load on different vehicles. Specifically, we follow the setup described in \cite{li2024geometry} and consider the ShapeNet car dataset and the Ahmed body dataset.
For both datasets, the three-dimensional Reynolds-averaged Navier-Stokes (RANS) equations
\begin{equation*} 
\begin{split}
(\bm{\overline{u}} \cdot \nabla)\bm{\overline{u}} + \frac{1}{\rho}\nabla \overline{p} + \nu \nabla^2 \bm{\overline{u}} + \nabla \cdot \bm{\tau}^R = 0, \\
\nabla \cdot \bm{\overline{u}} = 0    
\end{split}
\end{equation*} are solved around the vehicle. Here $\bm{u}$ denotes the time averaged flow velocity, $\rho$ is the constant flow density, $\overline{p}$ denotes the time averaged flow pressure, $\nu$ is the constant kinematic viscosity of the flow. And $\bm{\tau}^R$ denotes the  Reynolds stress tensor owing to the fluctuating velocity field, and this term is approximated by an appropriate turbulence model to close the RANS equations. On the vehicle, the no-slip velocity boundary condition $\bm{\overline{u}} = 0$ is applied and on the farfield, the constant inflow condition is applied. The primary goal is to estimate the time-averaged pressure field on the vehicle surface, a critical quantity in vehicle design.

For the ShapeNet car dataset \cite{umetani2018learning}, which was generated using a finite element Navier-Stokes solver \cite{taylor2024finite}, the inflow velocity is fixed at $72~\mathrm{km/h}$, corresponding to the Reynolds number about $Re = 5 \times 10^6$.
Various car shapes, such as sports cars, sedans, and SUVs, are collected from \cite{chang2015shapenet}, resulting in a total of $611$ samples. 
These car shapes are manually modified to remove side mirrors, spoilers, and tires, producing surface meshes with approximately $3.7k$ mesh points. The operator maps from the car's surface geometry $\Omega$ to the corresponding pressure load $\overline{p}$ on it:
\begin{equation*}
    \G^{\dagger} : \Omega \mapsto \overline{p}. 
\end{equation*}
For the Ahmed body dataset generated using a  GPU-accelerated OpenFoam solver \cite{jasak2007openfoam}, the inflow velocity ranges from $10~\mathrm{m/s}$ to $70~\mathrm{m/s}$ corresponding to the Reynolds number ranging from $4.35\times 10^5$ to $6.82\times 10^6$. The vehicle shapes are derived from the benchmark model described in \cite{ahmed1984some} designed by Ahmed, representing a ground vehicle with a bluff body. The dataset is created by systematically varying the vehicle’s length, width, height, ground clearance, slant angle, and fillet radius, resulting in a total of $551$ samples. 
The surface mesh is highly detailed, comprising approximately $100k$ mesh points. 
Unlike \cite{li2024geometry}, we simplify the input by removing parameters, such as length, width, height, ground clearance, slant angle, and fillet radius, as these are inherently represented within the point cloud data. We retain only the point cloud position along with the inflow velocity $ v $ and the Reynolds number $ Re $ as inputs. Thus, the operator now maps as follows:
\begin{equation*}
\G^\dagger : (v, Re, \Omega) \mapsto \overline{p}.
\end{equation*}

Next, we discuss the training process for the PCNO model.
For the ShapeNet car dataset, we retain 16 Fourier modes in each spatial direction, summing over $k = (k_1, k_2, k_3)$ where $-16 \leq k_i \leq 16$ in \eqref{eq:pcno-layer}. For the Ahmed body dataset, due to memory limitations, we retain only 8 Fourier modes in each spacial direction, summing over $k = (k_1, k_2, k_3)$ where $-8 \leq k_i \leq 8$ in \eqref{eq:pcno-layer}.
In both cases, we use a uniform density $\rho(x; \Omega) = \frac{1}{|\Omega|}$. The training datasets consist of 500 samples, with the remaining data used for test: 111 samples for the ShapeNet car dataset and 51 samples for the Ahmed body dataset. The training process takes approximately 75.54 seconds per epoch for the ShapeNet car dataset and 418.74 seconds per epoch for the Ahmed body dataset.

\begin{figure}
    \centering
    \includegraphics[width=0.42\linewidth]{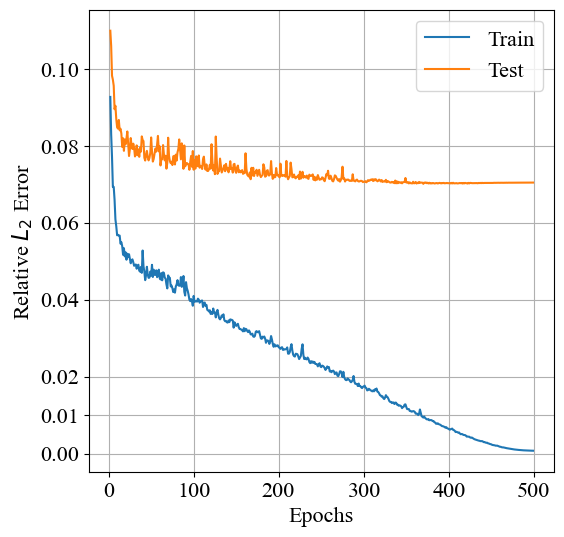}~~~~
    \includegraphics[width=0.39\linewidth]{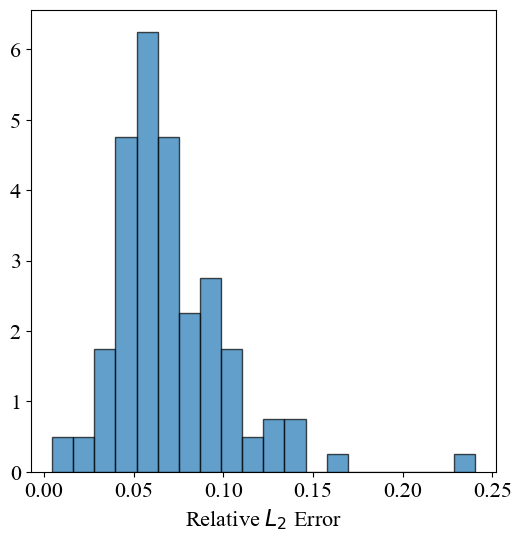}~~~
    \caption{ShapeNet car test: The relative training and test errors over iterations (left) and the distribution of the test errors (right). }
    \label{fig:car_shapenet_loss}
\end{figure}

\begin{figure}
    \centering
    \includegraphics[width=0.42\linewidth]{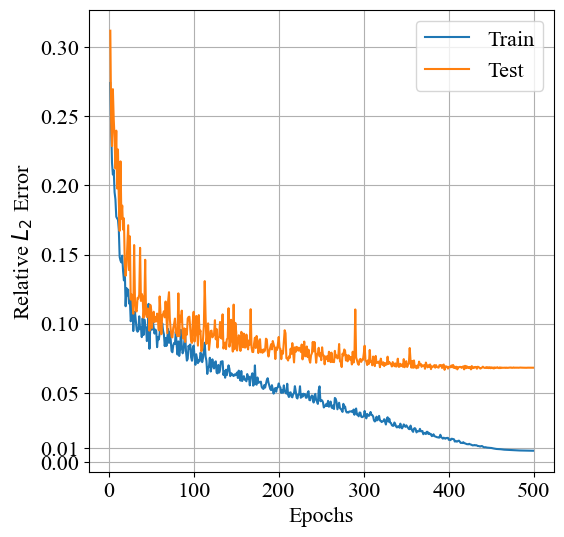}~~~
    \includegraphics[width=0.39\linewidth]{figures/vehicles_update/shapenet_loss_distribution.png}
    \caption{Ahmed body test: The relative training and test errors over iterations (left) and the distribution of the test errors (right). }
    \label{fig:ahmedbody_loss}
\end{figure}

The relative training and test errors over the iterations are visualized in  \cref{fig:car_shapenet_loss,fig:ahmedbody_loss} (left).
The resulting relative test errors are $7.04\%$ for the ShapeNet car dataset and $6.82\%$ for the Ahmed body dataset, outperforming the geometry-informed neural operator~(GINO) and the graph neural operator (GNO) results reported in \cite{li2024geometry}. However, a significant gap is observed between the training error and the test error, indicating potential overfitting due to the limited size of the training dataset.
The distributions of test errors are illustrated as bin plots in \cref{fig:car_shapenet_loss,fig:ahmedbody_loss} (right). For both datasets, notable outliers with large test errors are observed.
These outliers can likely be attributed to the limited diversity and size of the training dataset.
To further analyze the outliers, we visualize the test cases with the largest relative $L_2$ errors in \cref{fig:car_shapenet_pred,fig:ahmedbody_pred} (bottom). For the ShapeNet car dataset, the outlier is primarily due to the deep indentation at the front of the car. In the Ahmed body dataset, the outlier is characterized by a notably short length and large slant angle, combined with a smaller reference pressure load and significant error on the back face, resulting in the largest relative $L_2$ error.
Additionally, test cases with median relative $L_2$ errors are presented in the same figures (\cref{fig:car_shapenet_pred,fig:ahmedbody_pred} top). In these cases, the predictions align closely with the reference, suggesting that the model performs well on more representative samples. 
In terms of acceleration, the model preprocesses and predicts a single instance in 0.626 and 0.170 seconds, respectively, for the ShapeNet car dataset, and 10.976 and 0.172 seconds for the Ahmed body dataset. Therefore, we achieve significant speedup, as GPU-accelerated traditional solvers with parallelization take several hours to compute a single instance \cite{li2024geometry}.

\begin{figure}
    \centering
    \includegraphics[width=0.9\linewidth]{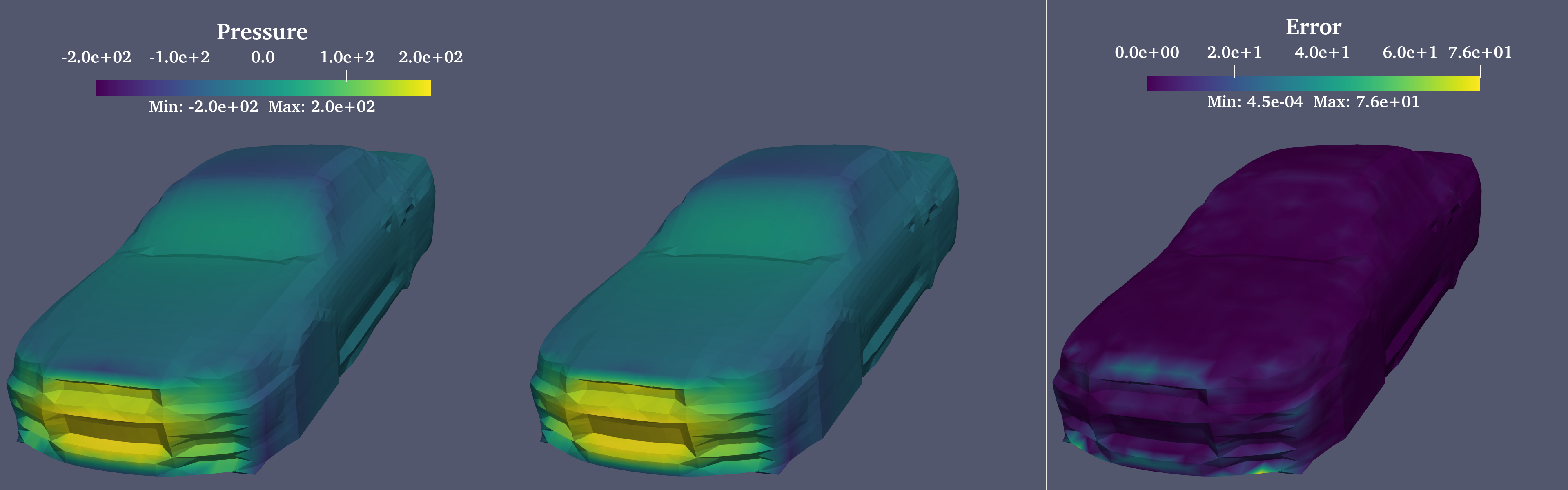}
    \includegraphics[width=0.9\linewidth]{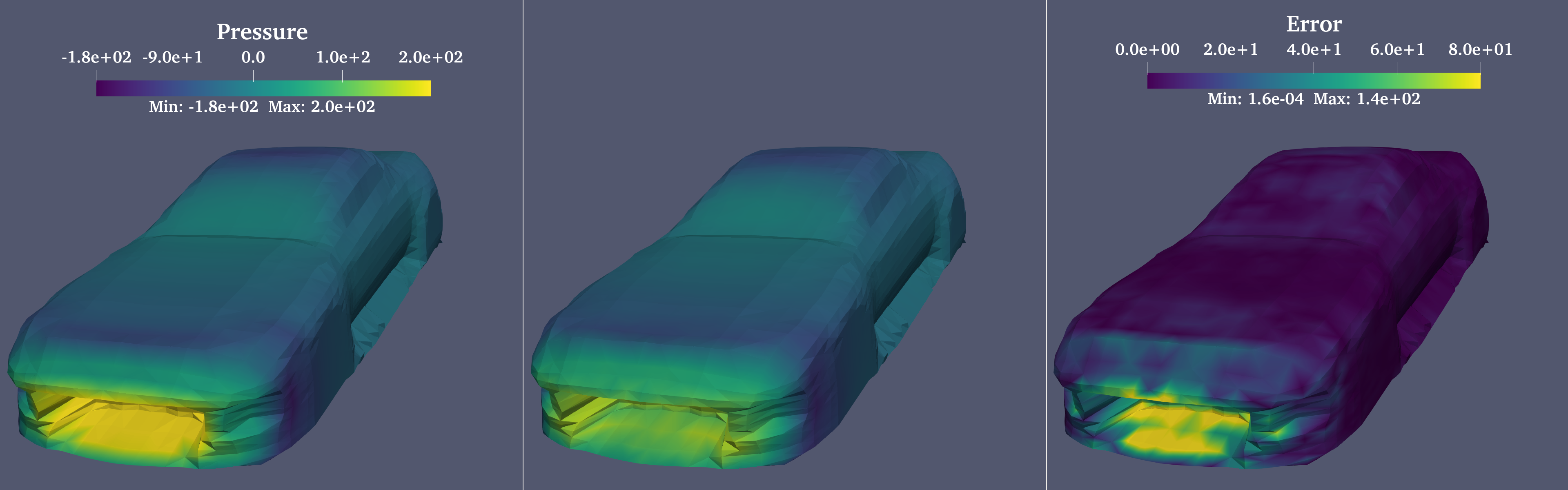}
    \caption{ShapeNet car test: From left to right, the panels display the reference surface pressure load $p(x)$ (Pa), the predicted surface pressure load, and the error (defined as the absolute value of difference between the reference and predicted solutions),  for the test cases with the median relative $L_2$ error (top) and the largest relative $L_2$ error (bottom).}
    \label{fig:car_shapenet_pred}
\end{figure}

\begin{figure}
    \centering
    \includegraphics[width=0.98\linewidth]{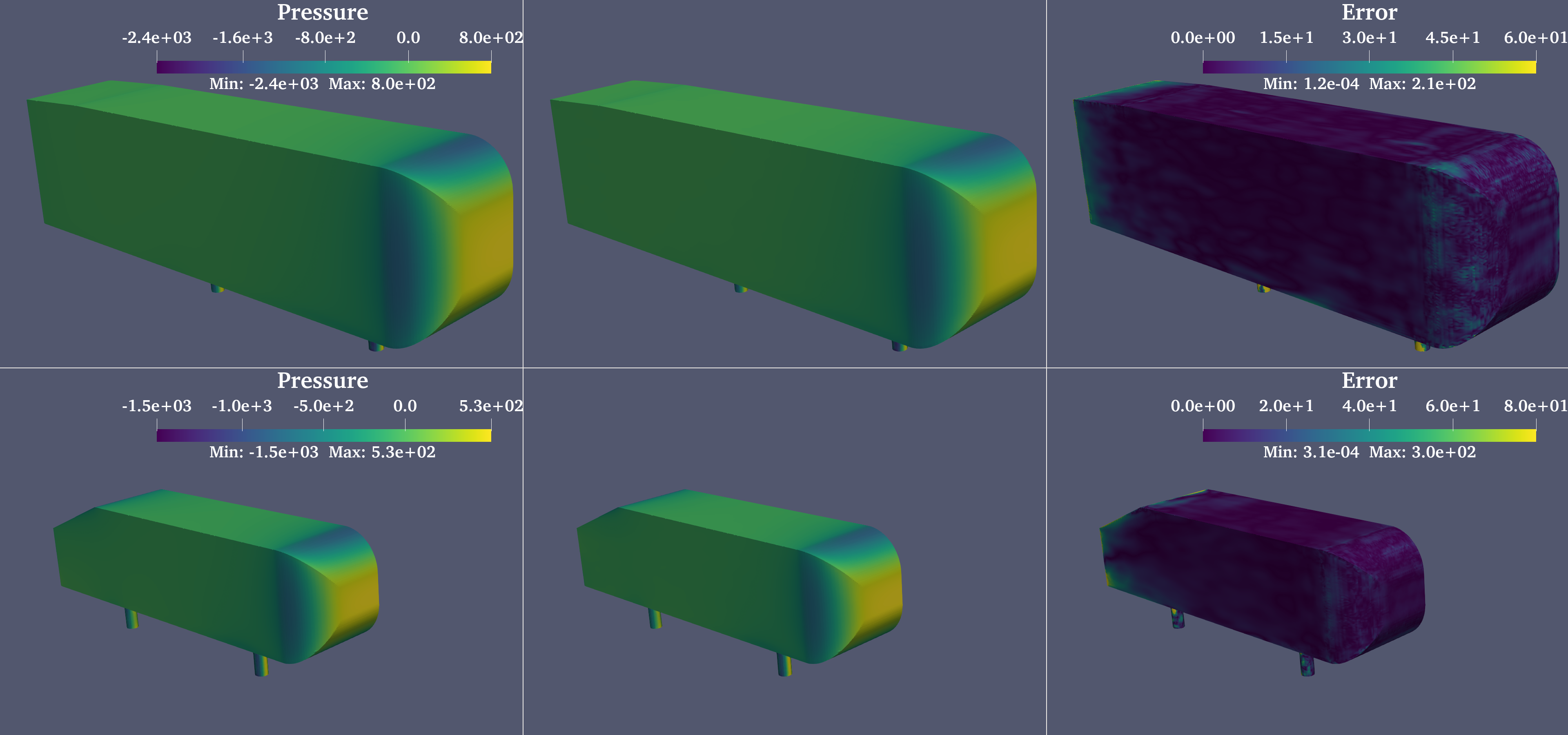}
    \caption{Ahmed body test: From left to right, the panels display the reference surface pressure load $p(x)$ (Pa), the predicted surface pressure load, and the error (defined as the absolute value of difference between the reference and predicted solutions),  for the test cases with the median relative $L_2$ error (top) and the largest relative $L_2$ error (bottom). }
    \label{fig:ahmedbody_pred}
\end{figure}

\subsection{Parachute Dynamics}
\label{ssec:parachute}
In this subsection, we consider predicting the inflation process of various parachutes under specified pressure loads. The dataset is generated by considering three types of parachutes: the disk-gap-band (DGB) parachute, a supersonic decelerator designed with a disk, gap, and band (See \cref{fig:DGB_loss}); the ringsail parachute, characterized by concentric, partially vented panels (resembling "rings") for high-performance deceleration (See \cref{fig:Arc_loss}); and the ribbon parachute, featuring longitudinal ribbons separated by reinforcement lines to enhance stability during supersonic deceleration (See \cref{fig:Robbin_loss}). 
For each type, 2,000 parachutes are generated by varying parameters such as the number of gores, diameter, gap size, and other specifications, as detailed in \cref{tab:parachute}.
The parachute inflates from its design shape under a pressure load applied to the canopy. The governing equations of the dynamic equilibrium for the structure are written in the Lagrangian formulation as
\begin{equation}
\label{eq:struct_equations}
     \rho \ddot{\bm{u}}- \nabla_{X}\cdot(\bm{F}\bm{S}) = \bm{f}_{ext}, \quad \textrm{in} \quad \Omega,
\end{equation}
where $\Omega$ represents the initial configuration of the parachute, described using material coordinates $X$. Here, $\rho$ is the density of the material, and $\bm{f}_{\text{ext}}$ denotes the external force vector acting on the parachute canopy due to the pressure load. The deformation gradient tensor is denoted by $\bm{F}$, and $\bm{S}$ represents the second Piola-Kirchhoff stress tensor.
Typically, the stress tensor is defined through a constitutive relation that relates it to the symmetric Green strain tensor:
\begin{equation*}
\bm{E} = \frac{1}{2}(\bm{F}^T \bm{F} - \bm{I}),
\end{equation*}
where $\bm{I}$ is the identity matrix.
The pressure load increases linearly from $0$ to $1000$~Pa over the first $0.1$ second and remains constant at $1000$~Pa thereafter. The parachute fabric, as well as the suspension and reinforcement line materials, are assumed to be linear elastic, with material properties consistent with those in \cite{huang2020modeling}.

\begin{table}
\centering
\begin{tabular}{|c|c|c|}
\hline
\multicolumn{2}{|c|}{Design variable} & \multicolumn{1}{c|}{Range} \\
\hline
\multicolumn{2}{|c|}{Number of gores} & \multicolumn{1}{c|}{$16, 20, \cdots, 36, 40$} \\

\multicolumn{2}{|c|}{Nominal diameter} & \multicolumn{1}{c|}{$5m \sim 9m$} \\

\multicolumn{2}{|c|}{Vent diameter} & \multicolumn{1}{c|}{$0.5m \sim 1m$} \\

\multicolumn{2}{|c|}{Suspension line length} & \multicolumn{1}{c|}{$10 m$} \\
\cline{1-3}
\multirow{2}{*}{DGB parachute} &  Canopy height & $1m \sim 2m$     \\  
&  Gap ratio & $20\%\sim30\%$     \\  
\cline{1-3}
\multirow{5}{*}{Ringsail parachute} &  Canopy height & $1m \sim 2m$     \\   
&  First ring bottom position  & $25\% \sim 30\%$       \\  
&  Second ring top position & $35\% \sim 40\%$     \\  
&  Second ring bottom position & $60\% \sim 65\%$   \\  
&  Third ring top position & $75\% \sim 80\%$  \\  
\cline{1-3}
\multirow{3}{*}{Ribbon parachute} &  Canopy height & $0m \sim 2m$     \\  
&  First reinforcement line position & $20\% \sim 40\%$       \\  
&  Second reinforcement line position  & $60\% \sim 80\%$     \\  
\hline
\end{tabular}
\caption{Parachute dynamics: Geometric parameters for each parachute design. The DGB parachute has a cylindrical design consisting of a disk, gap, and band. The ringsail parachute features a circular arc design composed of three ring structures, with the first ring attached to the vent. The top and bottom positions of these rings are parameterized by their relative distance from the vent's edge. The ribbon parachute is designed as a cone, incorporating two embedded longitudinal reinforcement lines. The suspension line has a fixed length of $10m$, connecting the fixed bottom point to the parachute canopy.}
\label{tab:parachute}
\end{table}

We generate the dataset by solving the structural governing equation \eqref{eq:struct_equations}, using AERO-Suite~\cite{farhat1998load,wang2011algorithms,farhat2010robust,huang2018family,borker2019mesh,michopoulos2024bottom} with the finite element method.
The vertical suspension lines and longitudinal reinforcement lines are modeled as one-dimensional beam elements, while the parachute canopy is modeled as two-dimensional shell elements.  
A uniform mesh size of approximately $0.1m$ is used, resulting in each parachute comprising about 6,000 mesh points.
Our objective is to learn the mapping from the initial design shape $\Omega$ to the parachute displacement fields at four specific time points during inflation: $t = 0.04$, $t = 0.08$, $t = 0.12$, and $t = 0.16$. These configurations capture the inflation process, where the parachute first inflates rapidly under pressure load, then overinflates, and finally bounces back. To accelerate online prediction and improve stability, we directly predict the solutions at these four moments, leaving the exploration of the time-dependent operator for future work.  The operator is defined  as follows:
\begin{equation*}
    \G^{\dagger} : (\mathds{1}_{\rm line}, \Omega) \mapsto (\bm{u}_{0.04}, \bm{u}_{0.08}, \bm{u}_{0.12}, \bm{u}_{0.16}), 
\end{equation*}
where $\mathds{1}_{\rm line}$ is an indicator function specifying whether a point lies on a one-dimensional element (i.e., suspension lines or longitudinal reinforcement lines). 

Next, we discuss the setup and training process for the PCNO model.
Since the parachute geometry $\Omega$ consists of components with different dimensionalities,  we define the density function using two integral domains. The one-dimensional domain, $\Omega_1$, includes the suspension lines and longitudinal reinforcement lines, while the two-dimensional domain, $\Omega_2$, represents the parachute fabric. These domains overlap in regions where the suspension and reinforcement lines are sewn to the canopy fabric along their alignment.
The integral operator in PCNO is expressed as
\begin{equation}
\begin{split}
\label{eq:integral-operator-rho2}
f_{\rm out}(x) =  \int_{\Omega_1} \kappa\bigl(x, y, f_{\rm in}\bigr) W_1^{v} f_{\rm in}(y) \rho_1(y; \Omega_1)  dy \
 + \int_{\Omega_2} \kappa\bigl(x, y, f_{\rm in}\bigr) W_2^{v} f_{\rm in}(y) \rho_2(y; \Omega_2)  dy,
\end{split}
\end{equation}
where $\rho_1(y; \Omega_1)$ and $\rho_2(y; \Omega_2)$ represent the distributions on $\Omega_1$ and $\Omega_2$, respectively. 
We use uniform distributions for both domains, defined as $\rho_1(y; \Omega_1) = \frac{1}{2|\Omega_1|}$ and $\rho_2(y; \Omega_2) = \frac{1}{2|\Omega_2|}$, ensuring that:
$$\int_{\Omega_1} \rho_1(y; \Omega_1) +  \int_{\Omega_2} \rho_2(y; \Omega_2) = 1.$$
The integral operator in \eqref{eq:integral-operator-rho2} combines contributions from both domains and employs Fourier kernels. Due to memory constraints, we retain 12 Fourier modes in each spatial direction for both domains, $\Omega_1$ and $\Omega_2$.
The PCNO model is trained on a dataset evenly distributed among three parachute types, with 333, 333, and 334 samples, respectively. The test dataset consists of 200 samples, also evenly distributed among the three parachute types.
The training process takes approximately 475.20 seconds per epoch.


The PCNO model achieves high accuracy, with a relative  $L_2$ test error of approximately $3.03\%$.
The relative training and test errors over iterations, as well as the distribution of test errors, are
visualized in \cref{fig:parachute loss compare}. 
During the early stages of training,  sudden and significant spikes in error are observed. This phenomenon may result from the large variations among different types of parachutes, with only around 300 data points for each, combined with numerical oscillations caused by a slightly high learning rate. After approximately 300 iterations, both the training and test errors begin to decline steadily, ultimately achieving high accuracy on the test dataset. 
The test error distributions indicate that ribbon parachutes exhibit slightly larger errors compared to the other two types. This can be attributed to the ribbon parachute's low porosity and relatively flat structure, 
 which tend to undergo larger displacements under pressure loading. These displacements are further visualized in \cref{fig:Robbin_loss}.
Furthermore, \cref{fig:DGB_loss,fig:Arc_loss,fig:Robbin_loss} visualize the PCNO model's performance across the three parachute types, highlighting test cases with median and largest relative $L_2$ errors, respectively. In these visualizations, the shape of the parachute represents the predicted solution during inflation, while the color intensity reflects the magnitude of the error between the actual and predicted displacement. 
Notably, the ribbon parachute in \cref{fig:Robbin_loss} exhibits the largest error, primarily due to its large motion at $t=0.08$ and $t=0.12$, which acts as an outlier. 
For most test cases, the PCNO model accurately captures the dynamic inflation patterns,  suspension line motion, and fabric folds of the parachute.
The pointwise displacement error is approximately $0.01m$, while the overall parachute length is around $10 m$ and the mesh resolution is around $0.1m$. This demonstrates the PCNO model's strong performance in maintaining low relative errors across the test cases. 
In terms of acceleration, the model preprocesses and predicts a single instance in approximately 1.15 and 0.159 seconds, respectively, compared to about one hour using traditional CPU-based solvers with parallelization.
These results highlight the potential of the PCNO model for addressing challenging real-world engineering problems involving complex geometries.

\begin{figure}
    \centering
    \includegraphics[width=0.39\linewidth]{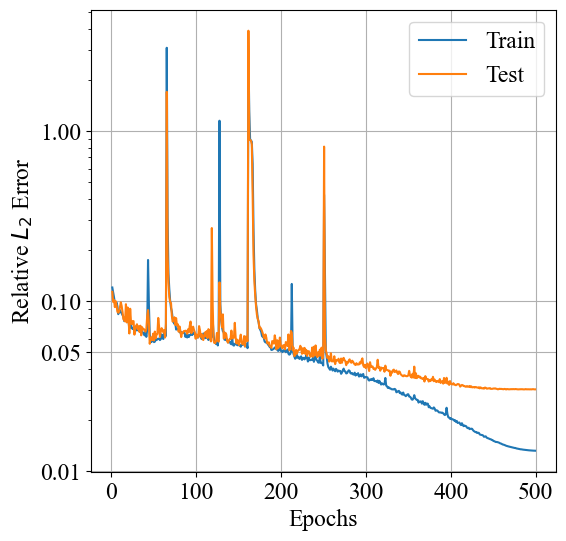}~~~~
    \includegraphics[width=0.355\linewidth]{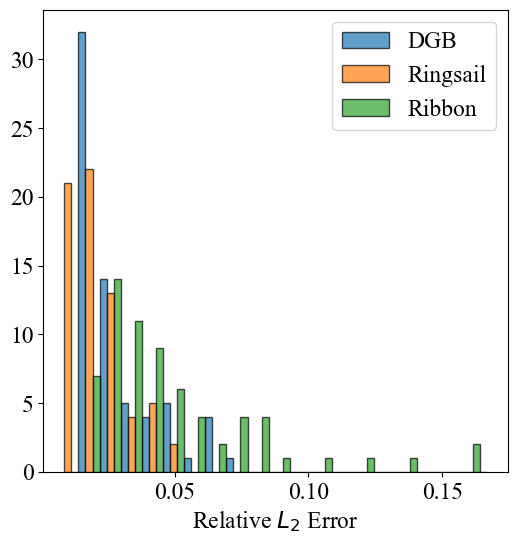}
    \caption{Parachute dynamics: The relative training and test errors over epochs (left) and the distribution of the test errors (right), visualized separately for the DGB, ringsail, and ribbon parachutes.}
    \label{fig:parachute loss compare}
\end{figure}

\begin{figure}
    \centering
    \includegraphics[width=0.98\linewidth]{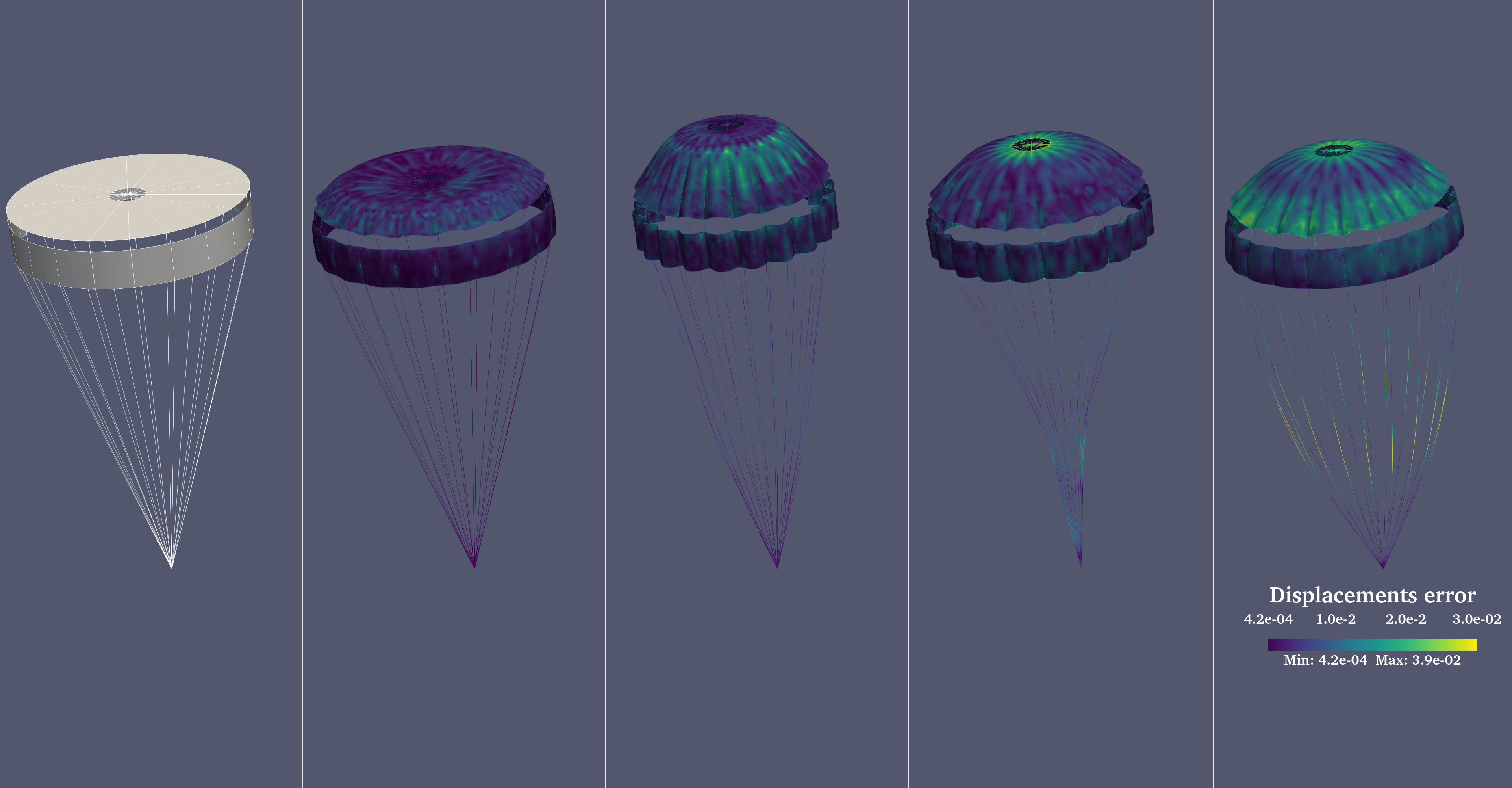}
    \includegraphics[width=0.98\linewidth]{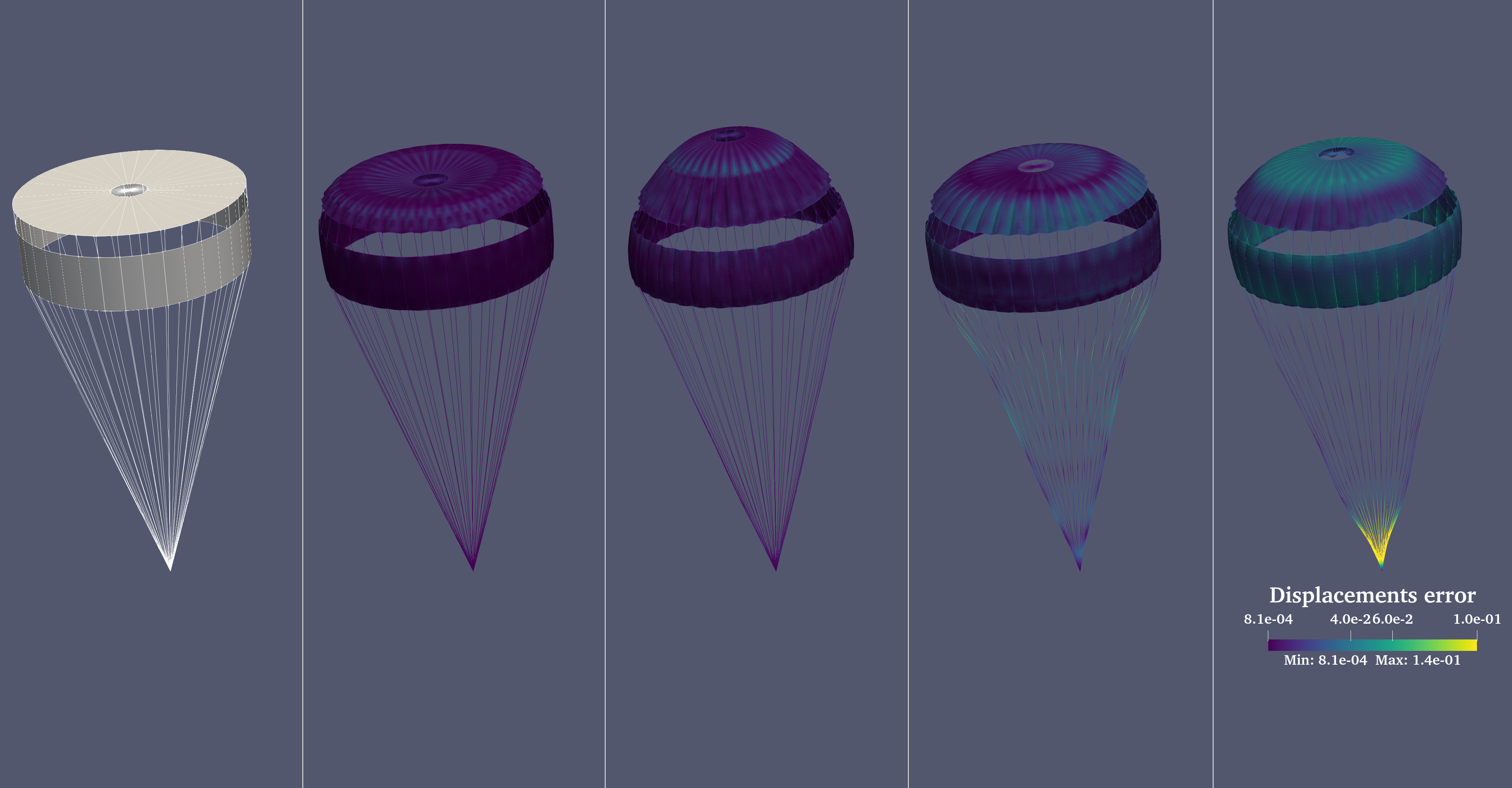}
    \caption{Parachute Dynamics: 
    From left to right, the panels display the design shape, and the inflation shape at $t=0.04$, $t=0.08$, $t=0.12$, and $t=0.16$ for the DGB parachute with the median relative $L_2$ error (top) and the largest relative $L_2$ error (bottom).
    The parachute shape represents the predicted solution during inflation, while the color intensity reflects the magnitude of the error between the reference and predicted displacement (m).  }
    \label{fig:DGB_loss}
\end{figure}

\begin{figure}
    \centering
    \includegraphics[width=0.98\linewidth]{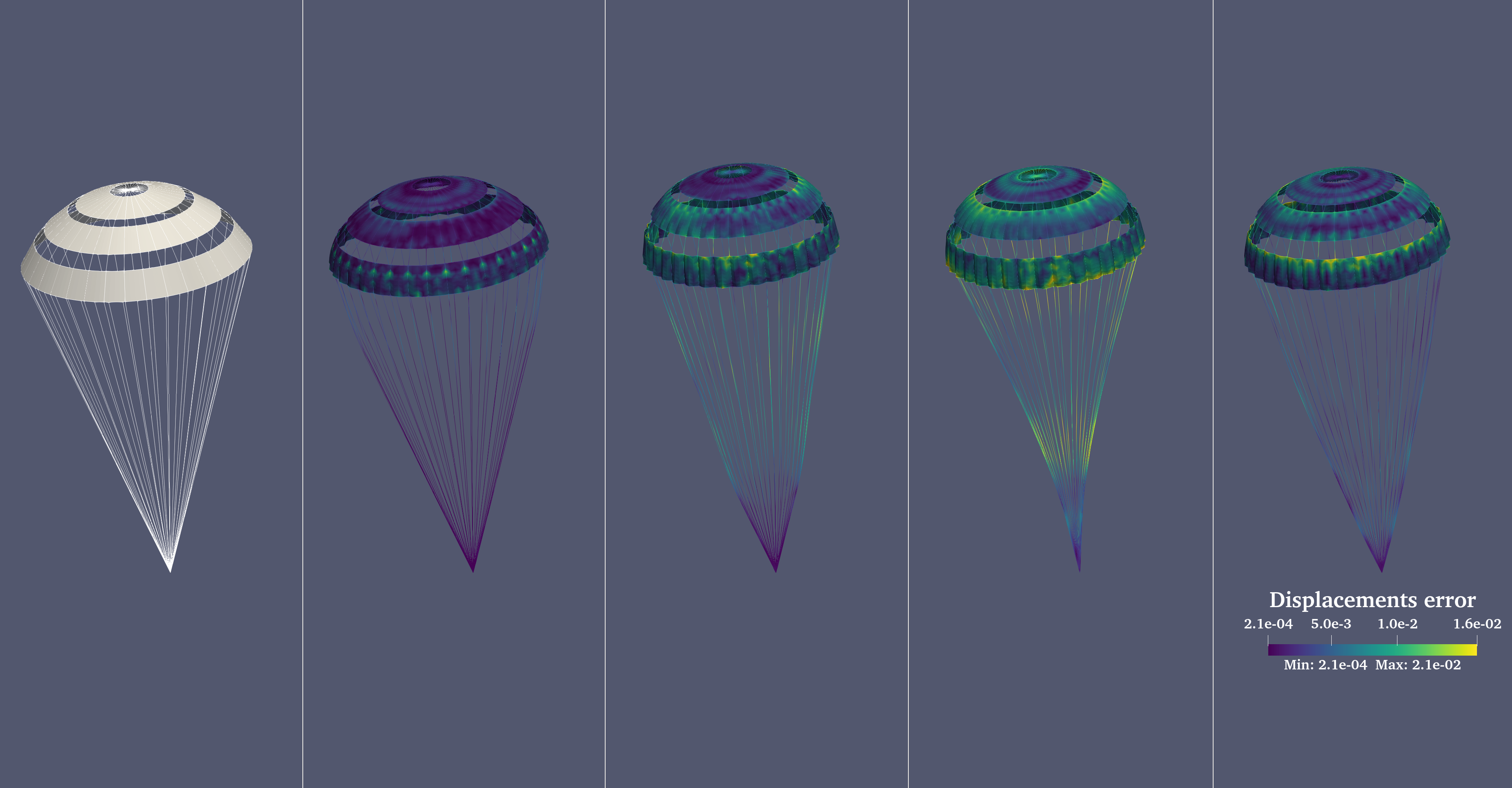}
    \includegraphics[width=0.98\linewidth]{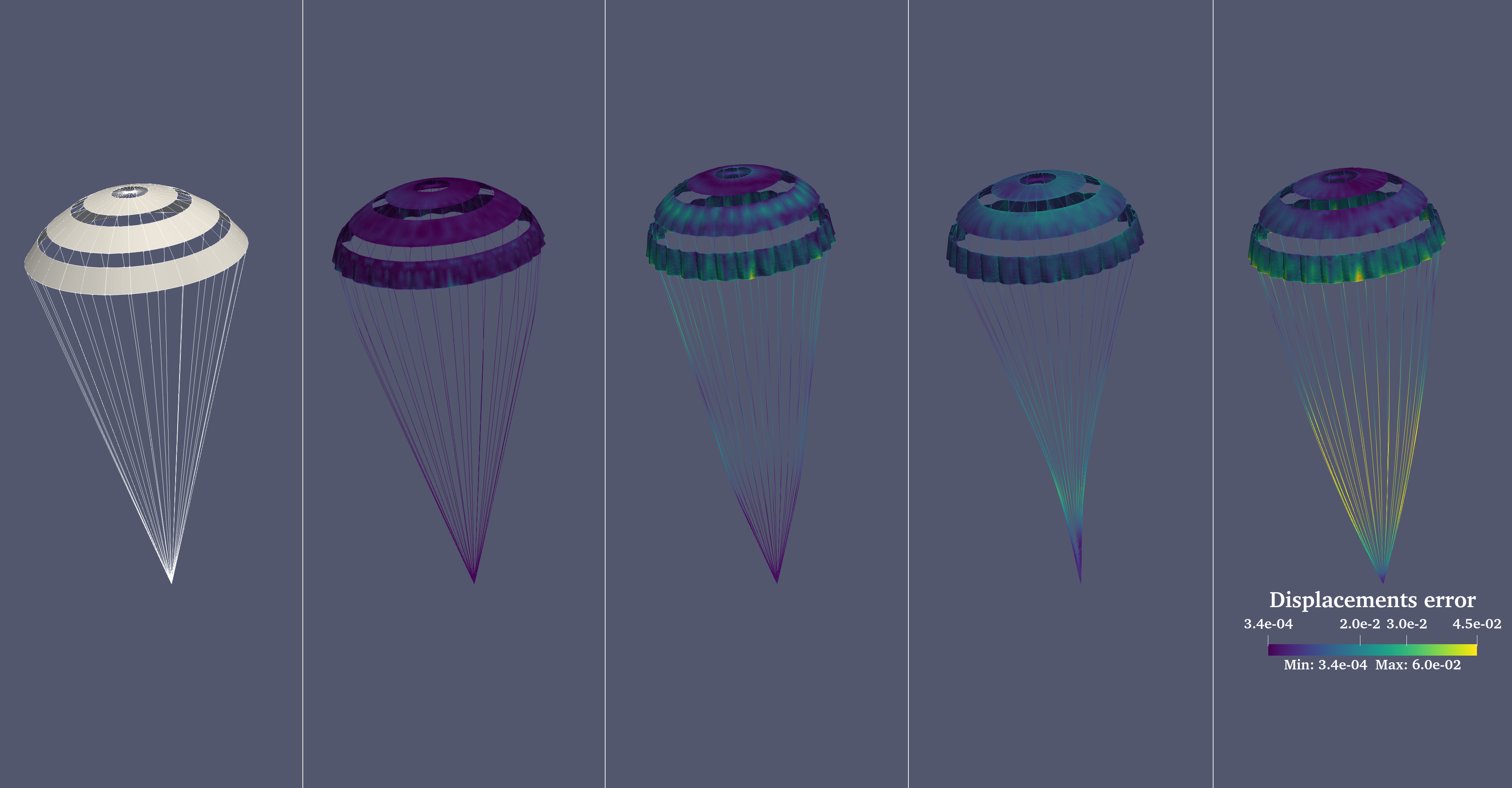}
    \caption{Parachute Dynamics: 
    From left to right, the panels display the design shape, and the inflation shape at $t=0.04$, $t=0.08$, $t=0.12$, and $t=0.16$ for the ringsail parachute with the median relative $L_2$ error (top) and the largest relative $L_2$ error (bottom).
    The parachute shape represents the predicted solution during inflation, while the color intensity reflects the magnitude of the error between the reference and predicted displacement (m). }
    \label{fig:Arc_loss}
\end{figure}

\begin{figure}
    \centering
    \includegraphics[width=0.98\linewidth]{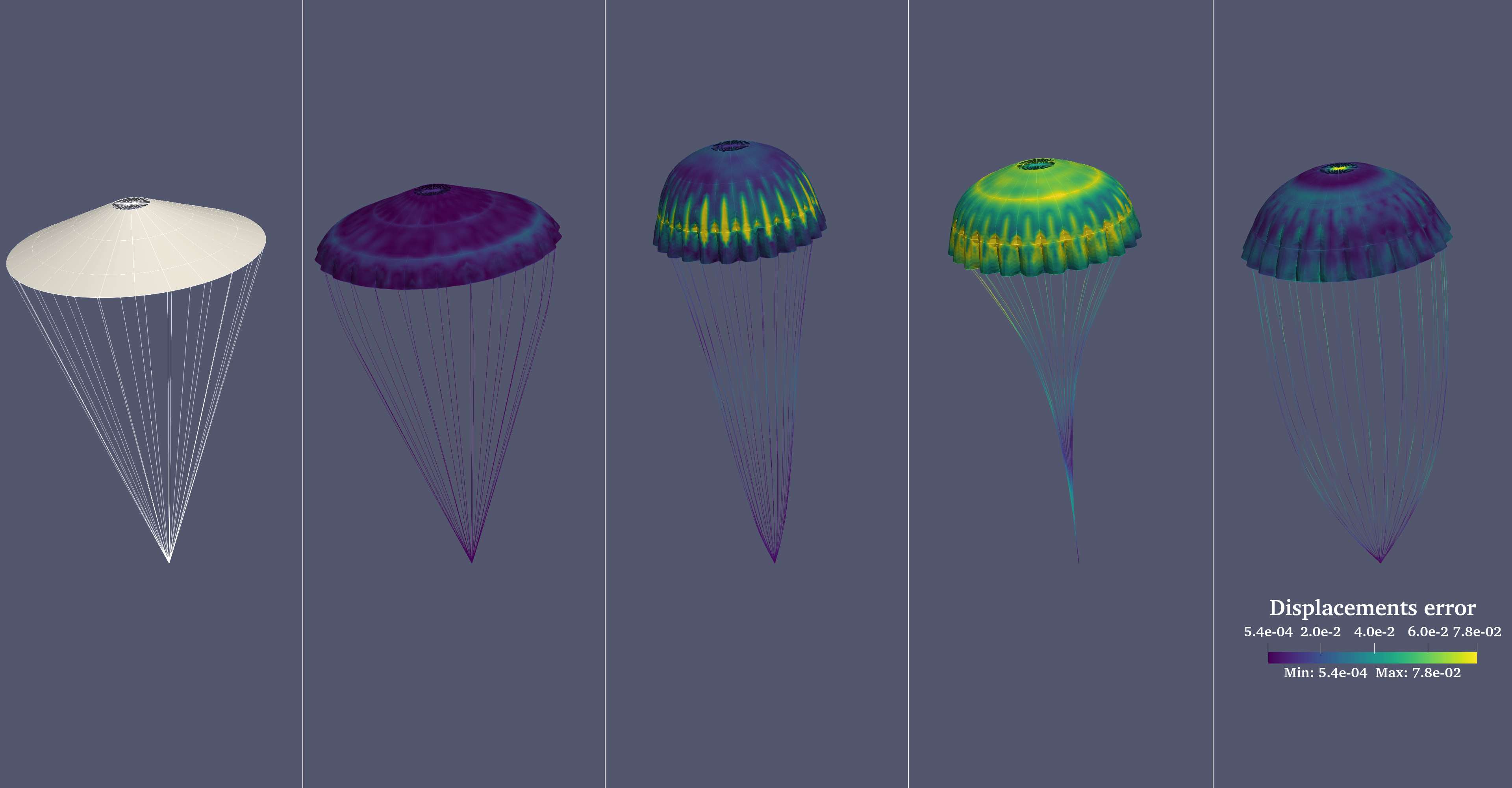}
    \includegraphics[width=0.98\linewidth]{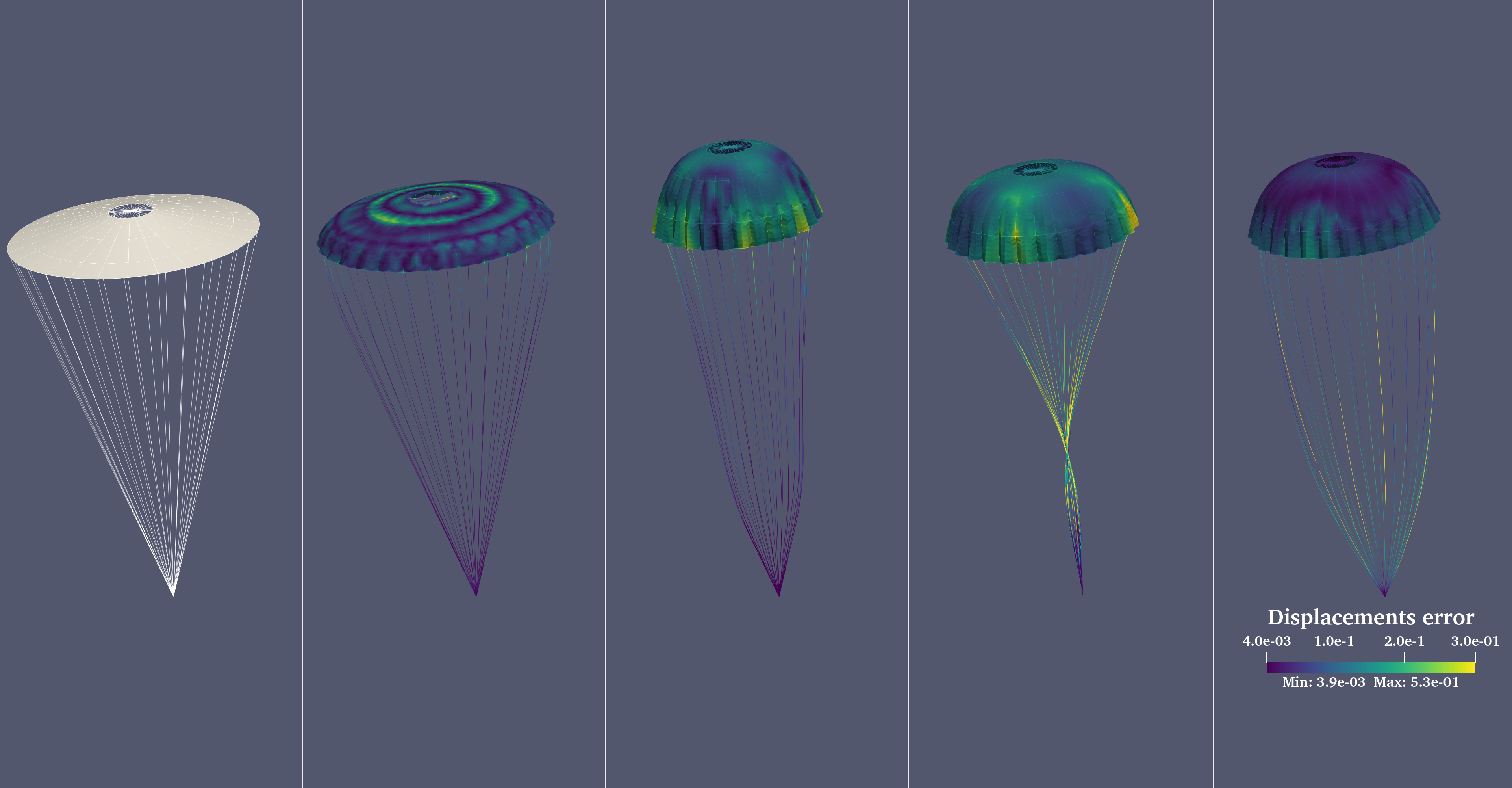}
    \caption{Parachute Dynamics: 
    From left to right, the panels display the design shape, and the inflation shape at $t=0.04$, $t=0.08$, $t=0.12$, and $t=0.16$ for the ribbon parachute with the median relative $L_2$ error (top) and the largest relative $L_2$ error (bottom).
    The parachute shape represents the predicted solution during inflation, while the color intensity reflects the magnitude of the error between the reference and predicted displacement (m). }
    \label{fig:Robbin_loss}
\end{figure}

\section{Conclusion}
\label{sec:conclusion}

In this paper, we introduce the Point Cloud Neural Operator (PCNO), a neural network-based surrogate model framework for approximating solution maps of parametric PDEs on complex and variable geometries. By designing the operator at the continuous level, and implementing it with traditional numerical methods, PCNO effectively handles point clouds and mitigates the impacts of discretization.  Its flexibility, effectiveness, and practicality are demonstrated through five PDE problems, including large-scale three-dimensional applications.
Several promising directions for future research remain. 
On the algorithmic side, systematically extending PCNO to time-dependent PDE systems would significantly broaden its applicability to more complex real-world scenarios. Additionally, developing a unified geometric foundation model that incorporates diverse data types could enhance its generalization capabilities. 
Improving the generalization capability from PDE solutions on simple geometries to complex geometry represents another promising research direction.
Furthermore, exploring PCNO’s potential for engineering design optimization, uncertainty quantification, and related applications is another valuable avenue.  
On the theoretical side, a rigorous analysis of sample complexity and the optimization process is essential to provide deeper insights into PCNO’s efficiency in practical applications.

\section*{Acknowledgments} We acknowledge the support of the National Natural Science Foundation of China through grant 12471403, the Fundamental Research Funds for the Central Universities, and the high-performance computing platform of Peking University. Yanshu Zhang also acknowledges the support of the elite undergraduate training program of School of Mathematical Sciences at Peking University. 

\appendix
\section{Mesh Size and Connectivity Computation}
\label{sec:mesh}
In this section, we provide details on computing mesh size $\dd \Omega_i$ and connectivity used in the integral and differential operators (See \cref{ssec:integral-op,ssec:differential-op}). 
For most PDE-related problems, the point cloud is typically associated with a mesh, from which the mesh size $\dd \Omega_i$ and connectivity can be estimated. 
For vertex-centered meshes, the quantities of interest are stored at each vertex, denoted as $\{x^{(i)}\}_{i=1}^{N}$.  
The mesh size at node $i$ is computed as:
\begin{equation}
    \dd\Omega_i = \sum_{j : e_j \textrm{ contains node } i} \frac{|e_j|}{n_{e_j}},
\end{equation}
where 
$e_j$ represents a cell containing node $i$,
$|e_j|$ is the volume of the cell, and  $n_{e_j}$  is the number of nodes in the cell. The volume of each cell is distributed equally among its nodes.
The connectivity of node $i$, defined as the set of its neighboring nodes, is given by:
\begin{equation}
    \cN(i) = \{ j: i,j \textrm{ belong to the same cell}\}.
\end{equation}
In other words, 
$\cN(i)$ consists of all nodes that share at least one cell with node $i$. 
For cell-centered meshes, the quantities of interest are stored at each cell. 
The mesh size at node $i$ is computed as the volume of the cell, it represents. 
The connectivity of node $i$, defined as the set of neighboring cells is given by:
\begin{equation}
    \cN(i) = \{ j:  \textrm{ cells } i,j \textrm{ share at least one node}\}.
\end{equation}
In other words, 
$\cN(i)$ includes all cells that share at least one node with cell $i$. 
When a mesh is unavailable, a Delaunay tessellation can be constructed from the point cloud, treating the resulting as a vertex-centered mesh.

\section{Proof of \cref{prop:uap}}
\label{sec:proof:uap}
In this section, we provide the proof of \cref{prop:uap}. First, we introduce two useful lemmas.
\begin{lemma}[Proposition A.8 in \cite{lanthaler2023nonlocal}]
\label{lem:expansion}
    Let $\G_B^\dagger: L_p(B, \R^{d_a+d+1})\to L_p(B,\R^{d_u})$ be a continuous operator. Let $\cK_B\subset L_p(B, \R^{d_a+d+1})$ be a compact set. Then for any $\epsilon>0$, there exist $\eta_1,\cdots,\eta_J\in L_p(B,\R^{d_u})$ and continuous functionals $\alpha_1,\cdots,\alpha_J: L_1(B,\R^{d_a+d+1})\to R$, such that
    \begin{equation}
        \sup_{\tilde{a}_B\in \cK_B} \left\|\G_B^\dagger (\tilde{a}_B) - \sum_{j=1}^J \alpha_j(\tilde{a}_B) \eta_j\right\|_{L_p(B,\R^{d_u})} \leq \epsilon.
    \end{equation}
\end{lemma}

\begin{lemma}[Modified Lemma A.9 in \cite{lanthaler2023nonlocal}]
\label{lem:alpha-approximation}
    Let $\alpha:L_1(B,\R^{d_a+d+1})\to R$ be a continuous nonlinear functional. Consider $\cK_B\subset L_1(B, \R^{d_a+d+1})$ as a compact set consisting of bounded functions, satisfying $\sup_{\tilde{a}_B \in \cK_B}\lVert \tilde{a}_B \rVert_{L_\infty} < \infty$. There exists a Point Cloud Neural Operator $\alpha_{\theta}: L_1(B, \R^{d_a+d+1}) \rightarrow L_1(B)$, whose output functions are constant and can therefore also be interpreted as a functional $\alpha_{\theta}: L_1(B, \R^{d_a+d+1}) \rightarrow \R$, such that 
    $$\sup_{\tilde{a}_B \in \cK_B}| \alpha(\tilde{a}_B) - \alpha_{\theta}(\tilde{a}_B) | \leq \epsilon,$$ 
    for any given $\epsilon$.
\end{lemma}

\begin{proof}
    Following the proof of Lemma A.9 in \cite{lanthaler2023nonlocal},  there exists an orthogonal basis, $\{\xi_i\}_{i=1}^\infty$, of $L_2(B,\R^{d_a+d+1})$,   which we may also assume to be in $C^\infty(B,\R^{d_a+d+1})$. Then we have uniform convergence
    \begin{equation}
    \label{eq:expansion-l2}
        \lim_{J \to \infty} \sup_{\tilde{a}_B\in \cK_B} \left\| \tilde{a}_B - \sum_{i=1}^J \langle \tilde{a}_B, \xi_j\rangle_{L_2} \xi_j \right\|_{L_1(B,\R^{d_a+d+1})} = 0 .
    \end{equation}
The claim relies on the fact that $B$ is bounded, $\cK_B \subset L_1 \cap L_{\infty} \subset L_2$, and $\cK_B \subset L_2$ is compact in $L_2(B,\R^{d_a+d+1})$. 
The approximation of the functional can be decomposed into two mappings
\begin{equation} 
\label{eq:two-mappings}
\tilde{a}_B \rightarrow \bigl(\langle \tilde{a}_B, \xi_1\rangle_{L_2}, \cdots, \langle \tilde{a}_B, \xi_J\rangle_{L_2}\bigr) \rightarrow \alpha\bigl(\sum_{i=1}^J \langle \tilde{a}_B, \xi_j\rangle_{L_2} \xi_j\bigr).
\end{equation}
The first mapping is continuous and, can be approximated by the proposed Point Cloud Neural Operator for any given $\epsilon$, such that 
\begin{equation} 
\label{dot-product-approximation}
\sup_{\tilde{a}_B  \in \cK_B}\Bigl|\langle \tilde{a}_B, \xi_j\rangle_{L_2} -   \int_{\Omega} \PP(\tilde{a}_B) \rho(x; \Omega) \dd x\Bigr|  =  \sup_{\tilde{a}_B \in \cK_B}\Bigl|\int_\Omega \tilde{a}_B \cdot \xi_j \dd x -   \int_{\Omega} \PP(\tilde{a}_B) \rho(x; \Omega) \dd x\Bigr| \leq \epsilon.
\end{equation}
Here $\PP$ represents the lift layer, consisting solely of local linear functions. The layer $\PP$ can be used to approximate the $L_p(B)$ function $\R^{d_a + d + 1} \rightarrow R$, $
\tilde{a}_B \rightarrow \frac{1}{\max\{\rho_B,\rho_{\inf}\}} \tilde{a}_B \cdot \xi_j(x) $ with error $\frac{\epsilon}{\sup_{\tilde{a}_B\in\cK_B} \lVert \rho_B \rVert_{L_q(B)}}$ where $\frac{1}{p} + \frac{1}{q} = 1$. Note that $\sup_{\tilde{a}_B\in\cK_B} \lVert \rho_B \rVert_{L_q(B)} < \infty$, as $\tilde{a}_B$ contains $\rho_B$ and is uniformly upper bounded, 
Furthermore, since  $\rho_B$  is lower bounded by $\rho_{\inf}$, we have the approximation 
$$
\sup_{\tilde{a}_B\in\cK_B}\Bigl\lVert \frac{\tilde{a}_B \cdot \xi_j}{ \rho_B }  -   \PP(\tilde{a}_B) \Bigr\rVert_{L_p(\Omega)} 
\leq 
\sup_{\tilde{a}_B\in\cK_B}\Bigl\lVert \frac{\tilde{a}_B \cdot \xi_j}{\max\{\rho_B,\rho_{\inf}\}}  -   \PP(\tilde{a}_B) \Bigr\rVert_{L_p(B)}  
\leq \frac{\epsilon}{\sup_{\tilde{a}_B\in\cK_B} \lVert \rho_B \rVert_{L_q(B)}}.
$$
By applying Hölder's inequality, we obtain \cref{dot-product-approximation}. The term $\int_{\Omega} \PP(\tilde{a}_B) \rho(x; \Omega) \dd x$ corresponds to the point cloud neural layer, which consists solely of a constant integral kernel $\kappa(x, y, f_{\rm  in}) = 1$, a special case of the Fourier kernel. Therefore, combining the lift layer with the point cloud neural layer gives the approximation of the first mapping in \eqref{eq:two-mappings}.

As for the second mapping in \eqref{eq:two-mappings}, we note that since $\tilde{a}_B \in \cK_B$ is upper bounded, the image of the first mapping is compact in $[-M,M]^J$ for sufficiently large $M > 0$.
The second mapping defines a continuous function $\beta : \R^J \to R$, $(c_1,\cdots, c_J) \mapsto \alpha(\sum_{j=1}^J c_j \xi_j)$, which can be approximated by the projection layer, on the compact domain $[-M,M]^J$. Here $\xi_j$ can be approximated using the lifting layer and then directed towards the final projection. Finally, by using the continuity of $\alpha$ and choosing sufficient large $J$, we prove the modified Lemma A.9.

\end{proof}

We are now prepared to prove Theorem \ref{prop:uap}.

\begin{proof}[Proof of Theorem \ref{prop:uap}]
    We demonstrate that the Point Cloud Neural Operator $\G_\theta$ , comprising a lifting layer (with several sublayers), a single point cloud neural layer, and a final projection layer (with several sublayers), is sufficient for approximation
    \begin{equation}
        \sup_{\tilde{a}_B \in \cK_B} \lVert \G^{\dagger}(\tilde{a}_B) - \G_{\theta}(\tilde{a}_B)\rVert_{L_{p}(B, \R^{d_u})} \leq \epsilon.
    \end{equation}
    And the restriction on $\Omega$ gives \eqref{eq:prop:uap}.

Following Lemma \ref{lem:expansion}, it suffices to approximate any continuous functional $\alpha: L_1(B,\R^{d_a+d+1})\to \R$ defined on $\cK_B$ and function $\eta \in L_p(B,\R^{d_u})$ by the proposed Point Cloud Neural Operator. For any $\alpha$, there exists a 
Point Cloud Neural Operator 
$\alpha_\theta$ given by the Lemma \ref{lem:alpha-approximation}, which approximates $\alpha$ uniformly. For $\eta$ as a $L_p$ function, the approximation is guaranteed by the classical approximation theory of neural networks.

Then, for any $j\in \{1, \cdots, J\}$ we aim to construct a Point Cloud Neural Operator $\G_{j,\theta}:\cK_B \to L_p(B, \R^{d_u})$ such that the following holds:
\begin{equation}
\label{eq:singlechannel}
    \sup_{\tilde{a}_B\in\cK_B} \Bigl\lVert \alpha_j(\tilde{a}_B) \eta_j - \G_{j,\theta}(\tilde{a}_B) \Bigr\rVert_{L_p(B, \R^{d_u})} \leq \epsilon / (2J).
\end{equation}

First, let $
M_\alpha = \sup_{\tilde{a}_B \in \cK_B}|\alpha_{j}(\tilde{a}_B)|$
and we can find a smoothed approximation $\tilde{\eta}_j\in C^{\infty}(B,\R^{d_u})$ satisfying\[
\|\tilde{\eta}_j - \eta_j\|_{L_p(B,\R^{d_u})}\leq \epsilon/(8JM_{\alpha}).
\]
Then we define $M_\eta=\|\tilde{\eta}_j\|_{L_p(B)}$. From Lemma \ref{lem:alpha-approximation} we know that there exists a Point Cloud Neural Operator $\alpha_{j,\theta}$ such that 
$$\sup_{\tilde{a}_B \in \cK_B}| \alpha_j(\tilde{a}_B) - \alpha_{j,\theta}(\tilde{a}_B) | \leq \epsilon / (8JM_\eta).$$
Also let $M_\alpha'=\sup_{\tilde{a}_B \in \cK_B}|\alpha_{j,\theta}(\tilde{a}_B)|$. Based on the approximation theory, there exists a neural network $\eta_{j,\theta}:B\to \R^{d_u}$ such that
\[
\|\tilde{\eta}_j-\eta_{j,\theta}\|_{L_p(B,\R^{d_u})}\leq \epsilon/(8JM_\alpha').
\]
Note that $\tilde{\eta}_j$ is in $C^{\infty}(B,\R^{d_u})$ so that the output of the neural network can be bounded, which is denoted as $M_{\eta}'=\|\eta_{j,\theta}\|_{L_\infty(B,\R^{d_u})}$. The Point Cloud Neural Operator is defined by combining $\alpha_{j,\theta}$ and $\eta_{j,\theta}$. Specifically,  $\eta_{j,\theta}$ can be placed in the lifting layer, with its output passed directly to the final layer. Then, the two operators $\alpha_{j,\theta}$ and $\eta_{j,\theta}$ can be combined with an additional final projection layer which approximates the multiplication:
\[
\sup_{x \in [-M_\alpha', M_\alpha'], y \in [-M_\eta',M_\eta']^{d_u}}
\|p_{\theta}(x,y) - xy\| \leq \delta.
\]
By the definitions of $M_\alpha'$ and $M_\eta'$, we know the output of $(\alpha_{j,\theta}, \eta_{j,\theta})$ locates in the cube $[-M_\alpha', M_\alpha'] \times [-M_\eta',M_\eta']^{d_u}$. Therefore, by choosing $\delta$ sufficiently small, we have\[
\|p_{\theta}(\alpha_{j,\theta}, \eta_{j,\theta}) - \alpha_{j,\theta}\eta_{j,\theta}\|_{L_p(B,\R^{d_u})} \leq \epsilon/(8J).
\]
Finally, the Point Cloud Neural Operator is defined as $\G_{j,\theta}(\tilde{a}_B) = p_{\theta}(\alpha_{j,\theta}(\tilde{a}_B), \eta_{j,\theta})$, and we have
\begin{equation}
\begin{aligned}
    \|\alpha_j(\tilde{a}_B) \eta_j - &\G_{j,\theta}(\tilde{a}_B)\|_{L_p(B,\R^{d_u})} \\
    \leq& \|\alpha_j(\tilde{a}_B) \eta_j - \alpha_j(\tilde{a}_B) \tilde{\eta}_j\|_{L_p(B,\R^{d_u})}+ \|\alpha_j(\tilde{a}_B) \tilde{\eta}_j - \alpha_{j,\theta}(\tilde{a}_B) \tilde{\eta}_j\|_{L_p(B,\R^{d_u})}\\
    &+ \|\alpha_{j,\theta}(\tilde{a}_B) \tilde{\eta}_j - \alpha_{j,\theta}(\tilde{a}_B) \eta_{j,\theta}\|_{L_p(B,\R^{d_u})}+ \|\alpha_{j,\theta}(\tilde{a}_B) \eta_{j,\theta} - \G_{j,\theta}(\tilde{a}_B)\|_{L_p(B,\R^{d_u})}\\
    \leq& M_\alpha \|\tilde{\eta}_j - \eta_j\|_{L_p(B,\R^{d_u})} + M_\eta | \alpha_j(\tilde{a}_B) - \alpha_{j,\theta}(\tilde{a}_B) |\\
    &+ M_\alpha' \|\tilde{\eta}_j-\eta_{j,\theta}\|_{L_p(B,\R^{d_u})} + \|\alpha_{j,\theta}(\tilde{a}_B)\eta_{j,\theta} - p_{\theta}(\alpha_{j,\theta}(\tilde{a}_B), \eta_{j,\theta})\|_{L_p(B,\R^{d_u})}\\
    \leq& \epsilon/(2J).
\end{aligned}
\end{equation}

Therefore, we successfully construct a Point Cloud Neural Operator $\G_{j,\theta}:\cK_B \to L_p(B, \R^{d_u})$ that satisfies \eqref{eq:singlechannel}, for any $j\in \{1, \cdots, J\}$. We then define a new Point Cloud Neural Operator $\G_\theta:=\sum_{j=1}^J \G_{j,\theta}$. Combining \eqref{eq:singlechannel} and Lemma \ref{lem:expansion} leads to
\begin{equation}
    \begin{aligned}\lVert \G_B^{\dagger}(\tilde{a}_B) - \G_{\theta}(\tilde{a}_B)\rVert_{L_p(B, \R^{d_u})}
        &\leq \left\lVert \G_B^{\dagger}(\tilde{a}_B) - \sum_{j=1}^J \alpha_j(\tilde{a}_B) \eta_j\right\rVert_{L_p(B, \R^{d_u})} + \left\lVert \sum_{j=1}^J \alpha_j(\tilde{a}_B) \eta_j - \G_{\theta}(\tilde{a}_B)\right\rVert_{L_p(B, \R^{d_u})}\\
        &\leq \left\lVert \G_B^{\dagger}(\tilde{a}_B) - \sum_{j=1}^J \alpha_j(\tilde{a}_B) \eta_j\right\rVert_{L_p(B, \R^{d_u})} + \sum_{j=1}^J \left\lVert \alpha_j(\tilde{a}_B) \eta_j - \G_{j,\theta}(\tilde{a}_B)\right\rVert_{L_p(B, \R^{d_u})}\\
        &\leq \epsilon/2 + J \cdot \epsilon / (2J)\\
        &= \epsilon.
    \end{aligned}
\end{equation}

\end{proof}

\section{Comparison Study}
\label{sec:comparison}

In this section, we present a detailed comparison of our method with other state-of-the-art neural operators on benchmark problems introduced in \cite{li2020fourier}: the Burgers' equation and the Darcy flow equation. 
The one-dimensional periodic Burgers' equation with viscosity $\nu=0.1$ is defined as follows:
\begin{equation}
        \begin{aligned}
            \frac{\partial}{\partial t}u(x,t)+\frac{1}{2}\frac{\partial}{\partial x}u^2(x,t) &=\nu \frac{\partial^2}{\partial x^2}u(x,t),    & &x\in(0,1),t\in(0,1], \\
            u(x,0) &=u_0(x),& &x\in(0,1).
        \end{aligned}
\end{equation}
The objective is to learn the operator $\mathcal{G}^{\dagger}:u_0(x)\mapsto u(x,1)$, which maps the random initial condition to the solution at time $t=1$.
The second problem is the two-dimensional Darcy flow equation:
\begin{equation}
        \begin{aligned}
            -\nabla \cdot (a(x)\nabla u(x)) &= 1, & & x\in (0,1)^2,\\
            u(x) &= 0, & & x\in \partial (0,1)^2.
        \end{aligned}
\end{equation}
Here, the goal is to learn the operator $\mathcal{G}^{\dagger}: a(x)\mapsto u(x)$, which maps the random discontinuous coefficient field $a(x)$ to the solution $u(x)$.
Although the computational domains for these problems are fixed unit cubes, and the point clouds are based on uniform grids, these benchmarks are widely used to evaluate neural operators.

The architecture of our Point Cloud Neural Operator remains consistent with the configuration described in \cref{sec:numerics}. For Burgers' equation, we retain 32 Fourier modes and train the model for 5000 epochs. For the Darcy flow equation, we retain 16 Fourier modes in each spatial direction and train for 1000 epochs. Following the experimental setup in \cite{li2020fourier}, we use 1000 samples for training and 200 samples for testing for both datasets, evaluating performance across various resolutions with different downsampling rates.
The relative $L_2$ test errors for these two problems are summarized in \cref{tab:burgers,tab:darcy}, respectively. 
For comparison, we also include results of other state-of-the-art neural operators, which report relative $L_2$ test errors on datasets of the same size. These methods include the FNO from ~\cite{li2020fourier}, the Factorized FNO (F-FNO) from \cite{tran2021factorized}, the Unet enhanced FNO (U-FNO) from \cite{wen2022u},  the Gabor-Filtered FNO (GFNO) from \cite{qi2024gabor}, the Galerkin Transformer (GT) from\cite{cao2021choose}, the Multigrid Neural Operator (trained with relative $H_1$ error), denoted as M{\scriptsize G}NO from \cite{he2023mgno}, and the Position-induced Transformer from \cite{junfengpositional}.
This comparison highlights that our proposed PCNO performs consistently across various resolutions without the need for hyperparameter tuning. Furthermore, despite not leveraging structural information inherent to these problems, PCNO achieves competitive performance compared to these neural operators that explicitly exploit such information. 
However, we also observed that PCNO requires more GPU memory and longer training times per epoch compared to similarly sized FNO models. This increased resource demand is attributed to the differential operator, which involves an additional message-passing step to compute the gradient. Enhancing the efficiency of our implementation will be a key focus of future work.

\begin{table}[htbp]
    \begin{center}
        \begin{tabular}{|c|cccc|}
        \hline
            \diagbox{Networks}{Resolutions}  & $256$ & $512$ & $1024$ & $2048$  \\ \hline
            FNO \tablefootnote{We train the FNO using the same configuration as PCNO, achieving errors that are an order of magnitude smaller than those reported in \cite{li2020fourier}.}  & 1.63  & 1.63  & 1.63  & 1.67 \\ \hline
            U-FNO \cite{qi2024gabor} & 0.91  & 0.88  & 0.72  & 0.60 \\ \hline
            GFNO \cite{qi2024gabor}  & \textbf{0.49}  & \textbf{0.49}  & \textbf{0.47}  & \textbf{0.48} \\ \hline
            GT \cite{cao2021choose} & --  & 1.20  & --  & 1.15 \\ \hline
            PCNO  &  \underline{0.67}  &  \underline{0.62}  & \underline{0.57}  &  \underline{0.55} \\
            \hline
        \end{tabular}
    \end{center}
    \caption{Burgers' equation: Relative $L_2$ error ($\times 10^{-3}$) between the reference and predicted results, with the best and second-best values highlighted in bold and underlined.}
    \label{tab:burgers}
\end{table}

\begin{table}[htbp]
    \begin{center}
        \begin{tabular}{|c|cccc|}
        \hline
            \diagbox{Networks}{Resolutions}  & $85\times85$ & $141\times141$ & $211\times211$ & $421\times421$\\ \hline
            FNO \cite{qi2024gabor}   & 0.78  & 0.71  & 0.73  & 0.74 \\ \hline
            F-FNO \cite{qi2024gabor} & 0.68  & 0.63  & 0.64  & 0.65 \\ \hline
            U-FNO \cite{qi2024gabor}& 0.67  & 0.64  & 0.74  & 0.78 \\ \hline
            GFNO  \cite{qi2024gabor}&  \underline{0.60}  & \underline{0.55}  & 0.50  & \underline{0.55} \\ \hline
            GT \cite{cao2021choose} & -- & 0.84   & 0.84  & -- \\ 
            \hline
            M{\scriptsize G}NO \cite{he2023mgno} & -- & --   & \textbf{0.18}  & -- \\ 
            \hline
            PiT \cite{junfengpositional} & -- & --   & 0.49  & -- \\ 
            \hline
            PCNO  &   \textbf{0.55}  &  \textbf{0.48}  &  \underline{0.46}  & \textbf{0.49} \\
            \hline
        \end{tabular}
    \end{center}
    \caption{Darcy flow equation:  Relative $L_2$ error ($\times 10^{-2}$) between the reference and predicted results, with the best and second-best values highlighted in bold and underlined.}
    \label{tab:darcy}
\end{table}

\section{Data Visualization}
\label{sec:visualization}
In this section, we present visualizations of the datasets generated in this work, which are used to evaluate PCNO. These datasets include samples from the Darcy flow problem (\cref{ssec:darcy}), the flow over an airfoil problem (\cref{ssec:airfoil_flap}), and parachute dynamics (\cref{ssec:parachute}). These visualizations are shown in \cref{fig:geometries}, illustrate the diverse geometric configurations associated with each problem.

\begin{figure}
    \centering
    \includegraphics[width=0.8\linewidth]{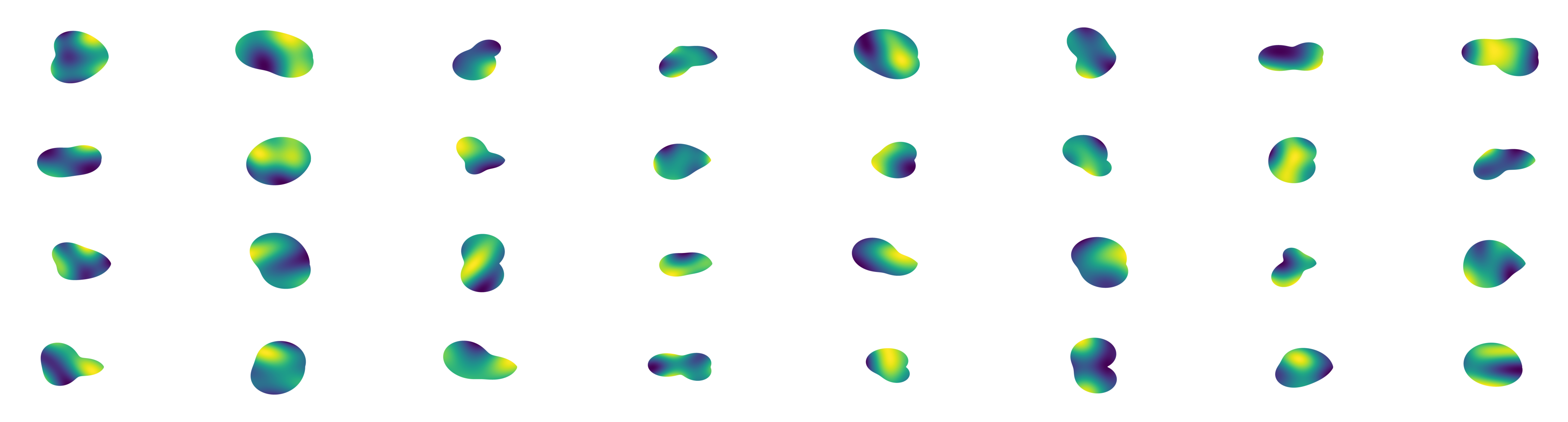}
     \includegraphics[width=0.8\linewidth]{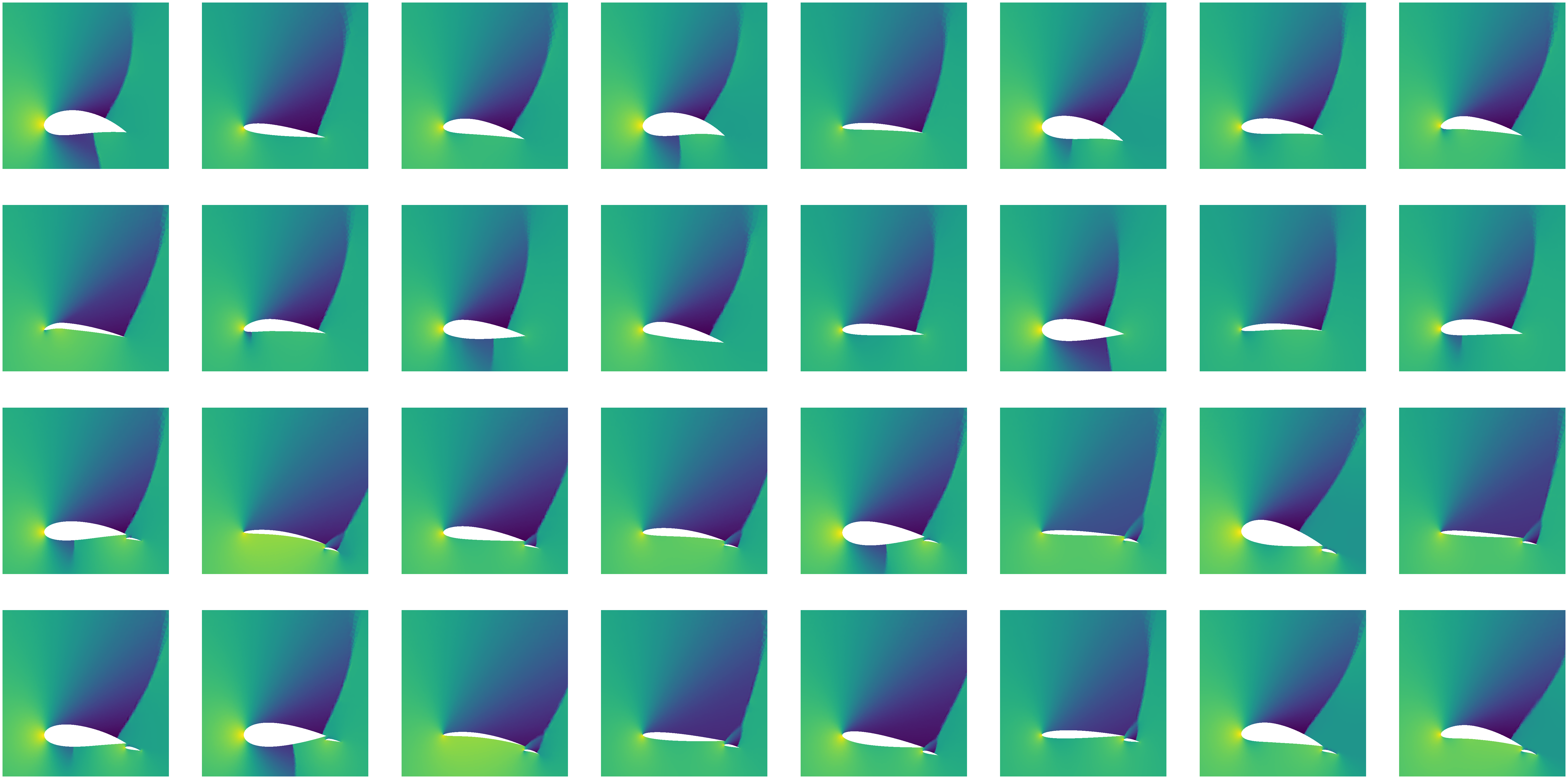}
     \includegraphics[width=0.8\linewidth]{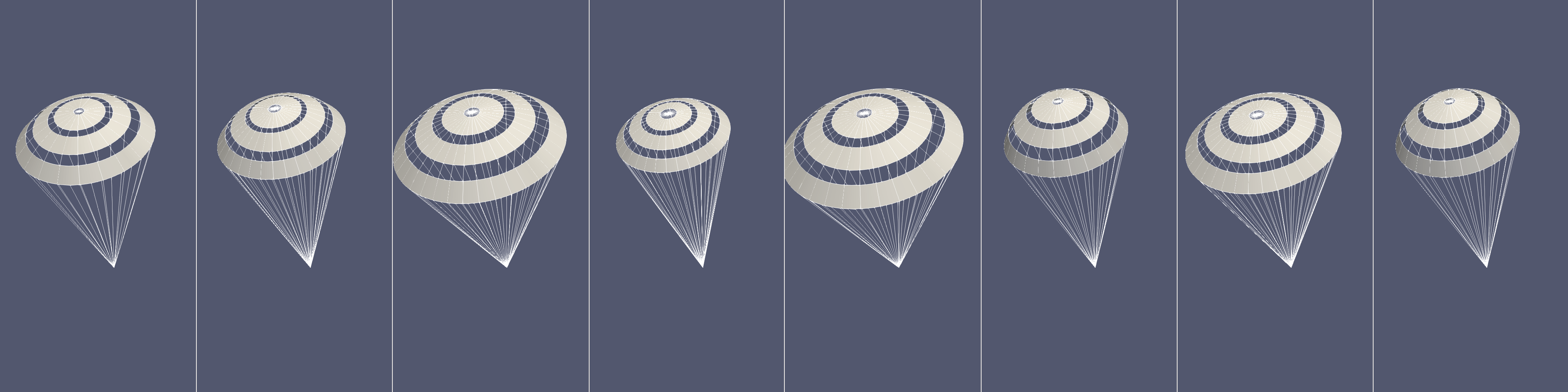}
    \includegraphics[width=0.8\linewidth]{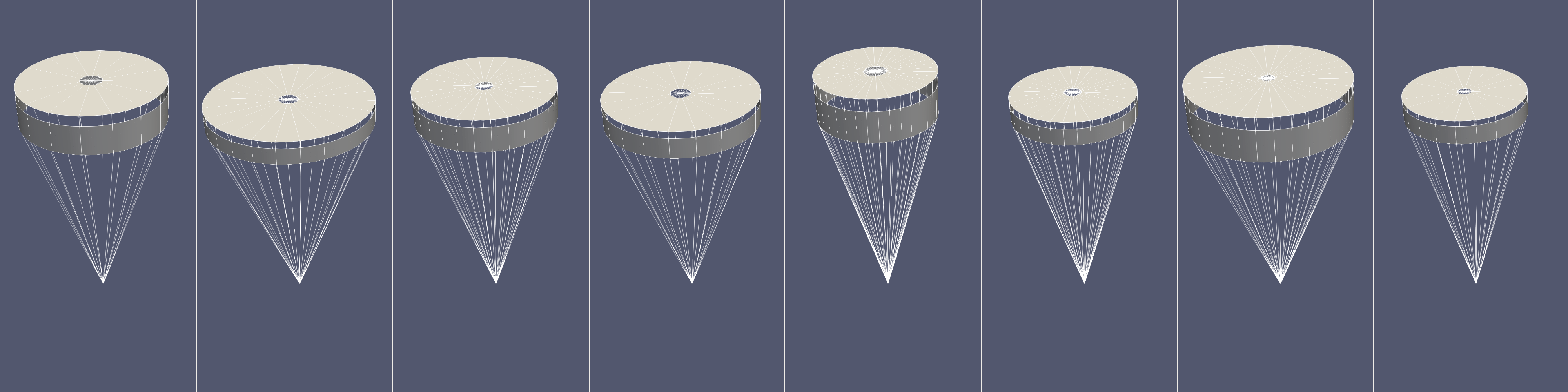}
    \includegraphics[width=0.8\linewidth]{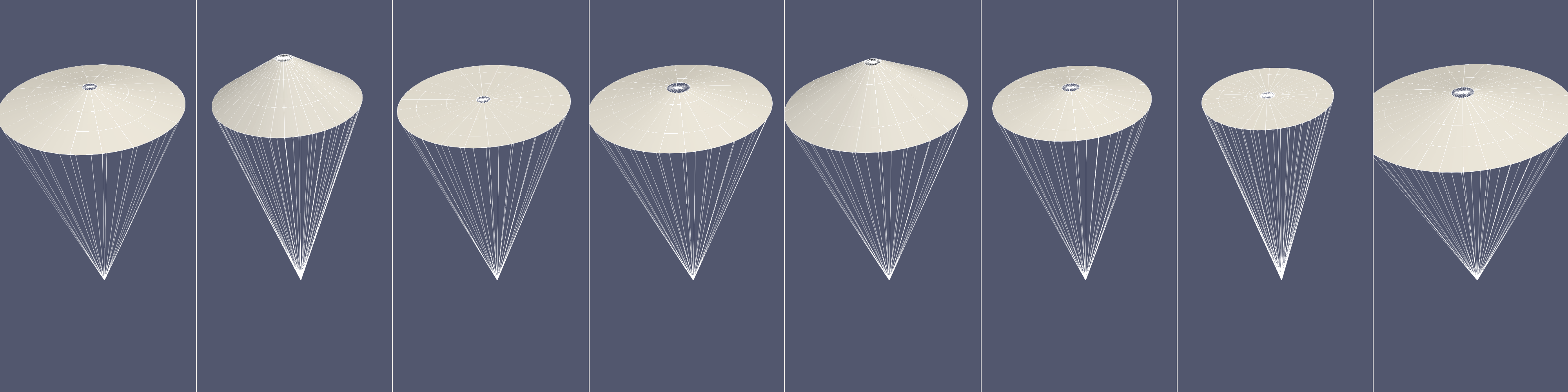}
    \caption{Sample visualizations from the dataset: (top to bottom) Darcy flow problem (\cref{ssec:darcy}), flow over an airfoil problem (\cref{ssec:airfoil_flap}), and parachute dynamics (\cref{ssec:parachute}).}
    \label{fig:geometries}
\end{figure}

\bibliographystyle{unsrt}
\bibliography{references}

\end{document}